\documentclass[amstex,12pt,reqno]{amsart}
\usepackage{amsmath, amssymb, amsthm}
\usepackage{mathrsfs}
\usepackage[all]{xy}
\usepackage{tikz}
\usetikzlibrary{intersections,calc,arrows.meta,patterns}

\newtheorem{theorem}{Theorem}[section]
\newtheorem{lemma}[theorem]{Lemma}
\newtheorem{proposition}[theorem]{Proposition}
\newtheorem{corollary}[theorem]{Corollary}
\newtheorem{claim}{Claim}

\theoremstyle{definition}
\newtheorem*{definition}{Definition}
\newtheorem*{remark}{Remark}

\title
[Parametrization of the $p$-Weil--Petersson curves]
{Parametrization of the $p$-Weil--Petersson curves: 
holomorphic dependence}

\author[H. Wei]{Huaying Wei} 
\address{Department of Mathematics and Statistics, Jiangsu Normal University \endgraf Xuzhou 221116, PR China} 
\curraddr{Department of Mathematics, School of Education, Waseda University \endgraf
Shinjuku, Tokyo 169-8050, Japan}
\email{hywei@jsnu.edu.cn} 

\author[K. Matsuzaki]{Katsuhiko Matsuzaki}
\address{Department of Mathematics, School of Education, Waseda University \endgraf
Shinjuku, Tokyo 169-8050, Japan}
\email{matsuzak@waseda.jp}

\makeatletter
\@namedef{subjclassname@2020}{%
\textup{2020} Mathematics Subject Classification}
\makeatother

\thanks{Research supported by 
Japan Society for the Promotion of Science (KAKENHI 18H01125 and 21F20027).}

\begin{document}

\maketitle

\begin{abstract}
Similarly to the Bers simultaneous uniformization,
the product of the $p$-Weil--Petersson Teichm\"uller spaces for $p \geq 1$
provides the coordinates for the space of $p$-Weil--Petersson embeddings $\gamma$ 
of the real line $\mathbb R$ into the complex plane $\mathbb C$. 
We prove the biholomorphic correspondence from this space to 
the $p$-Besov space of $u=\log \gamma'$ on $\mathbb R$ for $p>1$. 
From this fundamental result, several consequences follow immediately which 
clarify the analytic structures concerning parameter spaces of $p$-Weil--Petersson curves.
In particular, it follows that the correspondence 
of the Riemann mapping parameters
to the arc-length parameters keeping the images of curves
is a homeomorphism with bi-real-analytic dependence of change of parameters. This is a counterpart
to a classical theorem of Coifman and Meyer for chord-arc curves.
\end{abstract}

\section{Introduction}

The Weil--Petersson metric was originally introduced in the study of Teichm\"uller spaces of
Riemann surfaces. 
Besides this direction of researches of long history,
a study on the subspace $T_2$ of the universal Teichm\"uller space $T$ that admits the Weil--Petersson metric
was initiated by Cui \cite{Cu}, and later the Hilbert manifold structure of $T_2 \subset T$ and the curvatures 
of the Weil--Petersson metric were further investigated
by Takhtajan and Teo \cite{TT}.
Based on these fundamental work, Shen and his coauthors have developed
complex-analytic theories of the subspace $T_2$, which is nowadays called the
Weil--Petersson Teichm\"uller space. As Beltrami coefficients
representing elements of $T_2$ are square integrable with respect to the hyperbolic metric,
this is also called the integrable Teichm\"uller space. 

The theme of this paper is in complex analytic aspects of
the Weil--Petersson Teichm\"uller space.
In the series of papers, Shen \cite{Sh18} first characterized the Weil--Petersson class $W_2$,
which consists of
quasisymmetric homeomorphisms 
representing elements of $T_2$, without using quasiconformal extension. 
It was given in terms of
the fractional dimensional Sobolev space $H^{1/2}_{\mathbb R}$ of real-valued functions. 
Then, Shen and Tang \cite{ST}
regarded $H^{1/2}_{\mathbb R}$ as a new parameter space for $T_2$ which 
is real-analytically equivalent to the original complex Hilbert structure. In this work, 
they considered the Weil--Petersson class $W_2$ on the real line $\mathbb R$ and applied the arguments of 
chord-arc curves induced by BMO functions in Semmes \cite{Se}.
Further, Shen and Wu \cite{SW} considered the Weil--Petersson curves in the complex plane $\mathbb C$,
which is the complex generalization of the elements in $W_2$, and proved that
the Riemann mappings onto the domains defined by Weil--Petersson curves move continuously.

Recently, Bishop \cite{Bi} has accomplished a comprehensive study on the Weil--Petersson curves.
He collected about twenty characterizations of Weil--Petersson curves from various viewpoints of
analysis and geometry
including the complex analytic methods as we mentioned above.  
For instance, a planar geometric
characterization of a bounded Weil--Petersson curve $\Gamma$ is
given by the rate at which the perimeter of the inscribed $2^n$-polygon for $\Gamma$ converges to that of $\Gamma$.
Other characterizations are given by 
certain measurement of coarse smoothness for closed
rectifiable curves $\gamma$,
the M\"obius energy defined on $\gamma$ similarly to knots, hyperbolic geometry of convex cores spanned by $\gamma$,
the curvatures of minimal surfaces with the boundary $\gamma$, and so on.
One can find other related work in the references therein including Shanon and Mumford \cite{SM} on 2D-shape mapping,
and Wang \cite{Wa} arising from the SLE theory.

In this present paper, we give a foundation of the parametrization of the space of
Weil--Petersson curves in the framework of the quasiconformal Teichm\"uller theory.
We represent this space as the product of the Weil--Petersson Teichm\"uller spaces in three ways
and prove analytic and topological correspondences of those factors.
In this fashion, we can understand the structure of the space of
Weil--Petersson curves clearly and easily. In particular, we can reprove and extend several known results
that have important applications as immediate consequences from our fundamental theorems.
We also develop those arguments in the generalization to $p$-Weil--Petersson curves for $p>1$.

For $p \geq 1$, let ${\mathcal M}_p(\mathbb U)$ be the set of Beltrami coefficients
that are $p$-integrable with respect to the hyperbolic metric on the upper half-plane $\mathbb U \subset \mathbb C$.
The $p$-Weil--Petersson Teichm\"uller space $T_p(\mathbb U)$ is the quotient
space of ${\mathcal M}_p(\mathbb U)$ by the Teichm\"uller equivalence. 
Precise definitions are in Section 2 .
It has been proved that $T_p(\mathbb U)$ possesses a complex Banach manifold structure
via the Bers embedding. On the lower half-plane $\mathbb L \subset \mathbb C$, the corresponding spaces ${\mathcal M}_p(\mathbb L)$ and
$T_p(\mathbb L)$ are defined.
The $p$-Weil--Petersson class $W_p$ is the set of
quasisymmetric homeomorphisms $f:\mathbb R \to \mathbb R$ that extend
quasiconformally to $\mathbb U$ (and to $\mathbb L$) with their complex dilatations in ${\mathcal M}_p(\mathbb U)$ (and in ${\mathcal M}_p(\mathbb L)$).
Then, the element $f$ in $W_p$ for $p>1$ can be characterized by the property that
$f$ is locally absolutely continuous and $\log f'$ belongs to 
the $p$-Besov space $B_p^{\mathbb R}(\mathbb R)$ of real-valued functions, which coincides with $H^{1/2}_{\mathbb R}(\mathbb R)$ for $p=2$.
In fact, in our paper \cite{WM-3} generalizing \cite{ST}, it was proved that if a quasisymmetric homeomorphism $f$ of $\mathbb R$ is locally absolutely continuous
and $\log f'$ is in $B_p^{\mathbb R}(\mathbb R)$, then the variant of the Beurling--Ahlfors extension by the heat kernel
introduced by Fefferman, Kenig and Pipher \cite{FKP}
yields a quasiconformal homeomorphism of $\mathbb U$ whose complex dilatation is in ${\mathcal M}_p(\mathbb U)$.
We will elaborate these concepts in Section 3 as well as those introduced next.

A $p$-Weil--Petersson embedding $\gamma:\mathbb R \to \mathbb C$ is the restriction of
a quasiconformal homeomorphism of $\mathbb C$ whose complex dilatations on $\mathbb U$ and $\mathbb L$
belong to ${\mathcal M}_p(\mathbb U)$ and ${\mathcal M}_p(\mathbb L)$, respectively.
Due to this definition of a curve as a continuous mapping not only its image,
the space ${\rm WPC}_p$ of all normalized $p$-Weil--Petersson embeddings for $p \geq 1$ can be 
parametrized in the same spirit of the Bers simultaneous uniformization. Namely,
${\rm WPC}_p$ is identified with the product of the $p$-Weil--Petersson Teichm\"uller spaces
$T_p(\mathbb U) \times T_p(\mathbb L)$. 
Although this natural viewpoint has been missed in the literature, we emphasize in this paper that
this can significantly clarify the theory of ${\rm WPC}_p$. In our recent paper \cite{WM-0},
we carry out similar arguments for the space of chord-arc curves and succeed in obtaining 
several interesting consequences.

In Theorems \ref{holo} and \ref{biholo}, we prove the following basic result on this parametrization of ${\rm WPC}_p$.
The $p$-Besov space $B_p(\mathbb R)$ is a complex Banach space of complex-valued functions which contains $B_p^{\mathbb R}(\mathbb R)$
as the real subspace. The crucial point is not only to give the characterization of elements in ${\rm WPC}_p$ but also
to show that this correspondence is biholomorphic. This is a new result even in the case of $p=2$.

\begin{theorem}\label{thm1}
For any $p$-Weil--Petersson embedding $\gamma:\mathbb R \to \mathbb C$, 
the logarithm of its derivative $\log \gamma'$ belongs to
the $p$-Besov space $B_p(\mathbb R)$ for $p>1$. Moreover, this correspondence 
$$
L:{\rm WPC}_p \cong T_p(\mathbb U) \times T_p(\mathbb L) \to B_p(\mathbb R)
$$
is a biholomorphic homeomorphism onto the image.
\end{theorem}

The $p$-Weil--Petersson class $W_p$ of all normalized quasisymmetric homeomorphisms of $\mathbb R$ onto
itself is a real-analytic submanifold of ${\rm WPC}_p$ corresponding to the diagonal axis of the Bers coordinates
$T_p(\mathbb U) \times T_p(\mathbb L)$. All normalized $p$-Weil--Petersson embeddings that are induced by Riemann mappings
on $\mathbb U$ make a complex-analytic submanifold $RM_p$ of ${\rm WPC}_p$ corresponding to $\{[0]\} \times T_p(\mathbb L)$. 
Every $\gamma \in {\rm WPC}_p$ is represented uniquely as the reparametrization of $h \in RM_p$ by
$f \in W_p$. We set these correspondences as $f=\Pi(\gamma)$ and $h=\Phi(\gamma)$. In the Bers coordinates, these maps are defined by
$\Pi([\mu_1],[\mu_2])=([\mu_1],[\mu_1])$ and
$\Phi([\mu_1],[\mu_2])=([0],[\mu_2] \ast [\mu_1]^{-1})$. Here, we see that 
$\Phi$ is a continuous surjection
by the topological group property of $T_p$ (Theorem \ref{topgroup}). Thus, we have a homeomorphism
$$
(\Pi,\Phi):{\rm WPC}_p \to W_p \times RM_p.
$$
This is the second product structure of ${\rm WPC}_p$.

Moreover, let $IW_p$ be the space of all $p$-Weil--Petersson embeddings that have arc-length parametrizations.
Every $\gamma \in {\rm WPC}_p$ is represented uniquely as the reparametrization of its arc-length parametrization
$\gamma_0 \in IW_p$ by $f \in W_p$. 
In this way, ${\rm WPC}_p \cong W_p \times IW_p$. See Section 4 in more details.
Since the condition $\gamma \in IW_p$ is equivalent to the condition that $L(\gamma)$ is purely imaginary,
we have $IW_p=L^{-1}(iB_p^{\mathbb R}(\mathbb R)^\circ)$ and this is a real-analytic submanifold of ${\rm WPC}_p$.
Here, $iB_p^{\mathbb R}(\mathbb R)^\circ$ is the intersection of the open subset $L({\rm WPC}_p) \subset B_p(\mathbb R)$ with 
the real subspace $iB_p^{\mathbb R}(\mathbb R)$ consisting of purely imaginary functions. Clearly, 
$L(W_p)=B_p^{\mathbb R}(\mathbb R)$.
Then, there is a bijection $J$ from $B_p^{\mathbb R}(\mathbb R) \times iB_p^{\mathbb R}(\mathbb R)^\circ$ to
$L({\rm WPC}_p)$ such that this product structure is compatible with
$W_p \times IW_p$ on ${\rm WPC}_p$ under $J^{-1} \circ L$ (Lemma \ref{arclength}). 
Again by the topological group property of $T_p$,
$J$ is a homeomorphism (Theorem \ref{Jhomeo}) and thus we have a homeomorphism
$$
J^{-1} \circ L:{\rm WPC}_p \cong W_p \times IW_p\to B_p^{\mathbb R}(\mathbb R) \times iB_p^{\mathbb R}(\mathbb R)^\circ.
$$
This is the third product structure of ${\rm WPC}_p$.

By investigating these product structures,
we obtain topological correspondence between the factor subspaces of the products.
This in particular shows that these subspaces admit different analytic structures that are topologically equivalent to each other.
We develop those arguments in Section 6 and the results are summarized as follows.

\begin{theorem}\label{thm3}
The space ${\rm WPC}_p$ for $p>1$
admits two other product structures $W_p \times RM_p$ and $W_p \times IW_p$
that are homeomorphic to ${\rm WPC}_p$.
The fiber structures for the projections to the second factors in the both products
are the same and each fiber consists of the family of all normalized $p$-Weil--Petersson 
embeddings of the same image.
\end{theorem}

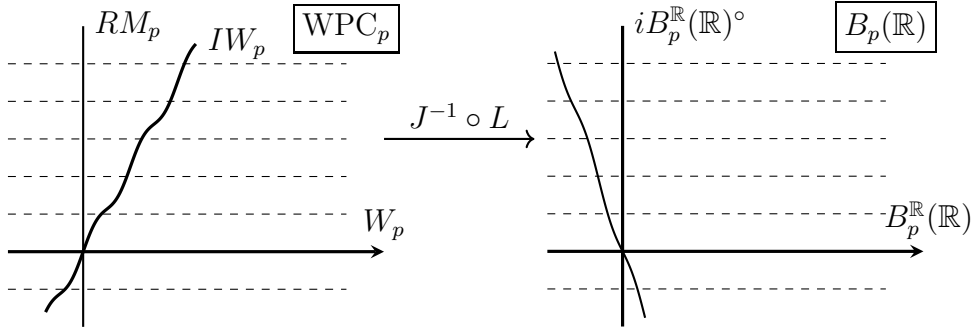
\begin{figure}[h]
\begin{tikzpicture}
\draw[->,>=stealth,very thick] (-1,0)--(4,0)node[above]{$W_p$}; 
\draw[-,>=stealth,thick] (0,-1)--(0,3)node[right]{$RM_p$}; 
\draw[dashed] (-1,-0.5)--(3.5,-0.5); 
\draw[dashed] (-1,0.5)--(3.5,0.5); 
\draw[dashed] (-1,1)--(3.5,1); 
\draw[dashed] (-1,1.5)--(3.5,1.5); 
\draw[dashed] (-1,2)--(3.5,2); 
\draw[dashed] (-1,2.5)--(3.5,2.5)node[above]{\fbox{${\rm WPC}_p$}}; 
\draw[very thick,samples=100,domain=-0.5:1.5] plot(\x,{1.8*\x+0.1*sin(10*\x r)})node[right]{$IW_p$};

\draw[->,thick](4,1.5)--(6,1.5); 
\draw(5,1.5)node[above]{$J^{-1} \circ L$};
\end{tikzpicture}
\begin{tikzpicture}
\draw[->,>=stealth,very thick] (-1,0)--(4,0)node[above]{$\ B_p^{\mathbb R}(\mathbb R)$}; 
\draw[-,>=stealth,very thick] (0,-1)--(0,3)node[right]{$iB_p^{\mathbb R}(\mathbb R)^\circ$}; 
\draw[dashed] (-1,-0.5)--(3.5,-0.5); 
\draw[dashed] (-1,0.5)--(3.5,0.5); 
\draw[dashed] (-1,1)--(3.5,1); 
\draw[dashed] (-1,1.5)--(3.5,1.5); 
\draw[dashed] (-1,2)--(3.5,2); 
\draw[dashed] (-1,2.5)--(3.5,2.5)node[above]{\fbox{$B_p(\mathbb R)$}}; 
\draw[thick,samples=100,domain=-0.9:0.3] plot(\x,{-3*\x+0.1*sin(10*\x r)});
\end{tikzpicture}
\caption{The fiber structure of ${\rm WPC}_p$}
\end{figure}

More detailed real-analytic correspondence between the subspaces of ${\rm WPC}_p$
are also obtained. Especially, the dependence of the Riemann mappings on the arc-length parameters of $p$-Weil--Petersson
curves is an interesting problem, and this is considered in Corollary \ref{real-analytic} and Theorem \ref{CM}.
The latter theorem contains the following new achievement on 
the real-analytic dependence of the following reparametrization map.

\begin{theorem}\label{thmnew}
The reparametrization map from the arc-length parametrization to the Riemann mapping parametrization
defined by 
$$
\lambda=L \circ \Pi \circ L^{-1}|_{iB_p^{\mathbb R}(\mathbb R)^\circ}:iB_p^{\mathbb R}(\mathbb R)^\circ \to B_p^{\mathbb R}(\mathbb R)
$$ 
is a real-analytic homeomorphism onto an open contractible domain of $B_p^{\mathbb R}(\mathbb R)$ for $p>1$ whose inverse is
also real-analytic.
\end{theorem}

This is the Weil--Petersson curve version of the original result for
chord-arc curves by Coifman and Meyer \cite{CM}. Their researches provided
an important topic to this field. See 
Semmes \cite[Section 6]{Se} and Wu \cite{Wu}.
Our formulation of the space of $p$-Weil--Petersson curves can make these arguments transparent, and
in particular, the real-analycity of $\lambda$ already follows from our arguments immediately.
The essential step in the original work is the investigation of the inverse correspondence, but
once we know that $\lambda$ is real-analytic, we can apply their result to see that
$\lambda^{-1}$ is also real-analytic. These are demonstrated in Section 5.

\section{The $p$-Weil--Petersson Teichm\"uller space and the $p$-Besov space}

A measurable function $\mu$ on $\mathbb U$ is called a {\it Beltrami coefficient} if $\Vert \mu \Vert_\infty<1$.
By the solution of the Beltrami equation, there exists a quasiconformal homeomorphism $F$ of
$\mathbb U$ onto itself whose complex dilatation $\mu_F=\bar \partial F/\partial F$ coincides with $\mu$
uniquely up to the post-composition of affine transformations of $\mathbb U$.
The definition on the lower half-plane $\mathbb L$ can be similarly done for this and all other concepts 
appearing hereafter.

For $p \geq 1$, let $\mathcal M_p(\mathbb U)$ be the set of Beltrami coefficients $\mu$ satisfying
$$
\Vert \mu \Vert_p^p=\int_{\mathbb U}\frac{|\mu(z)|^p}{y^2} dxdy<\infty.
$$
By the norm $\Vert \cdot \Vert_\infty+\Vert \cdot \Vert_p$, 
$\mathcal M_p(\mathbb U)$ is a domain of the corresponding Banach space.

The {\it $p$-Weil--Petersson Teichm\"uller space} $T_p(\mathbb U)$ on the upper half-plane $\mathbb U$ is
defined to be
the set of all Teichm\"uller equivalence classes $[\mu]$ for $\mu \in \mathcal M_p(\mathbb U)$.
Here, $\mu$ and $\mu'$ are equivalent if the quasiconformal homeomorphisms $F_\mu$ and $F_{\mu'}$ 
of $\mathbb U$ onto itself determined by $\mu$ and $\mu'$ have the same boundary extension to $\mathbb R$
up to the post-composition with affine transformations of $\mathbb R$.
The quotient map $\pi:\mathcal M_p(\mathbb U) \to T_p(\mathbb U)$ taking the equivalence class
is called the {\it Teichm\"uller projection}.
The canonical complex Banach manifold structure of $T_p(\mathbb U)$ for $p \geq 1$ 
is introduced via
the Bers embedding into certain complex Banach space (see \cite[Theorem 4.4]{Ya}, \cite[Theorem 2.1]{TS}, \cite[Theorem 4.1]{WM-1},
and Appendix).

The boundary extension to $\mathbb R$ of a quasiconformal homeomorphism $F:\mathbb U \to \mathbb U$ 
is called a {\it quasisymmetric homeomorphism}. If such an $f:\mathbb R \to \mathbb R$
is the extension of $F$ whose complex dilatation is in $\mathcal M_p(\mathbb U)$, 
we say that $f$ is a {\it $p$-Weil--Petersson class homeomorphism}. Then, the above definition of $T_p(\mathbb U)$
is equivalent to saying that $T_p(\mathbb U)$ is the set of all $p$-Weil--Petersson class homeomorphisms
modulo affine transformations of $\mathbb R$. 

A $p$-Weil--Petersson class homeomorphism for $p>1$ can be intrinsically defined as an increasing homeomorphism 
$f:\mathbb R \to \mathbb R$ such that $f$ is locally absolutely continuous and $\log f'$ belongs to
$B_p^{\mathbb R}(\mathbb R)$ defined below (see \cite[Theorem 1.1]{Sh18}, \cite[Theorem 1.2]{TS}, 
\cite[Theorems 1.2 and 1.3]{STW} and \cite[Theorem 5.5]{WM-1}). Partially this will be seen in Lemma \ref{real-1}. 
 
The {\it $p$-Besov space} $B_p(\mathbb R)$ for $p > 1$ is
the set of all locally integrable complex-valued functions $u$ on $\mathbb R$ that satisfy
$$
\Vert u \Vert^p_{B_p}=\frac{1}{4\pi^2}\int_{-\infty}^{\infty}\!\int_{-\infty}^{\infty} \frac{|u(t)-u(s)|^p}{|t-s|^2} dsdt<\infty.
$$
It is easy to see that if $u \in B_p(\mathbb R)$ then $|u|,\ {\rm Re}\,u,\ {\rm Im}\,u \in B_p(\mathbb R)$.
We can regard $B_p(\mathbb R)$ as a complex Banach space with norm $\Vert \cdot \Vert_{B_p}$
by taking the quotient of constant functions. In other words, we regard $B_p(\mathbb R)$ as
a homogeneous Besov space, which is often denoted by $\dot B_p(\mathbb R)$ in the literature.
The real Banach subspace of $B_p(\mathbb R)$ consisting of the elements represented by 
real-valued functions is denoted by $B_p^{\mathbb R}(\mathbb R)$.
We note that when $p=1$, $B_p(\mathbb R)$ degenerates into the space of constant functions (see \cite[Exercise 17.14]{Leo}).

Concerning the variable change operator on this space, 
the following result was shown in \cite[Theorem 12 and Remark 5]{B2} (see also \cite[Theorem 2.2]{Vo} and \cite[Theorem 1.3 and Section 3.4]{BS}). 

\begin{proposition}\label{pullback} Let $p>1$.
The increasing homeomorphism $h$ from $\mathbb R$ onto itself is quasisymmetric if and only if 
the variable change operator $P_h: u \mapsto u\circ h$ gives an isomorphism of the $p$-Besov space $B_p(\mathbb R)$
onto itself,
that is, $P_h$ and $(P_h)^{-1}$ are bounded linear operators. 
\end{proposition}

\begin{remark}
In the case of $p=2$, it is known that the operator norm $\Vert P_h \Vert$ of the variable change operator $P_h:B_p(\mathbb R) \to B_p(\mathbb R)$ depends only on the doubling constant of $h$ 
(or equivalently, the quasisymmetric constant of $h$ or the Teichm\"uller distance $d_\infty(h,{\rm id})$).
See \cite[Theorem 3.1]{NS}. On the contrary, in the case of $p \neq 2$, the estimate of the operator norm seems
more difficult. See \cite[Remark 4]{B2}. In a special case where $\log h'$ belongs to $B_p(\mathbb R)$,
a certain dependence of $\Vert P_h \Vert$ on $h$ will be shown later in Lemma \ref{general}.
\end{remark}

Next, we consider analytic function spaces.
Let $\mathcal B(\mathbb U)$ denote the {\it Bloch space} of functions $\varphi$ holomorphic on $\mathbb U$ with semi-norm 
$$
\Vert \varphi \Vert_{\mathcal B} = \sup_{z \in \mathbb U} |\varphi'(z)|y. 
$$
Let $\mathcal B_p(\mathbb U)$ denote 
the {\it analytic $p$-Besov space} for $p > 1$ (or $p$-Dirichlet space) of holomorphic functions $\varphi$ on $\mathbb U$ with semi-norm 
$$
\Vert \varphi \Vert_{\mathcal B_p} = \left(\frac{1}{\pi}\iint_{\mathbb U} |\varphi'(z)|^p y^{p-2} dxdy \right)^{\frac{1}{p}}. 
$$
Then, $\mathcal B_p(\mathbb U) \subset \mathcal B_q(\mathbb U) 
\subset \mathcal B(\mathbb U)$ 
for $1 < p \leq q$, and the inclusion maps are continuous. 
By considering functions in these spaces modulo additive constants, which we always do hereafter,
the semi-norms become norms and the spaces become complex Banach spaces.

These spaces can be defined on the unit disk $\mathbb D$ in the same way as on $\mathbb U$, and 
any conformal homeomorphism $\mathbb U \to \mathbb D$ induces an isometric isomorphism between the corresponding spaces.
For example,
we take the Cayley transformation $T(z) = (z-i)/(z+i)$, which maps $\overline{\mathbb U}$ onto $\overline{\mathbb D}$, 
and define the push-forward operator $T_*: \varphi \mapsto \varphi\circ T^{-1}$ for functions $\varphi$ on $\mathbb U$. Then,
$$
\frac{1}{\pi}\iint_{\mathbb D}|(T_*\varphi)'(w)|^p\left(\frac{1-|w|^2}{2}\right)^{p-2}dudv
=\frac{1}{\pi}\iint_{\mathbb U} |\varphi'(z)|^p y^{p-2}dxdy,
$$
and by defining ${\mathcal B}_p(\mathbb D)$ as the space of holomorphic functions on $\mathbb D$ with the
norm in the left side integral of the above equation finite, ${\mathcal B}_p(\mathbb D)$ and ${\mathcal B}_p(\mathbb U)$
are isometric.

We also see that the Cayley transformation $T$ induces an isometric isomorphism
between $B_p(\mathbb S)$ and $B_p(\mathbb R)$. Here, $B_p(\mathbb S)$ can be also defined in a natural way. Indeed,
\begin{equation*}
\begin{split}
&\quad \frac{1}{4\pi^2}\int_{\mathbb S} \int_{\mathbb S} \frac{|(T_*\varphi)(w_1)-(T_*\varphi)(w_2)|^p}{|w_1-w_2|^2}|dw_1||dw_2|\\
&=\frac{1}{4\pi^2}\int_{\mathbb R} \int_{\mathbb R} 
\frac{|\varphi(z_1)-\varphi(z_2)|^p}{|T(z_1)-T(z_2)|^2}|T'(z_1)||T'(z_2)||dz_1||dz_2|\\
&=\frac{1}{4\pi^2}\int_{\mathbb R} \int_{\mathbb R} 
\frac{|\varphi(z_1)-\varphi(z_2)|^p}{|z_1-z_2|^2}|dz_1||dz_2|,
\end{split}
\end{equation*}
where we used the identity 
$$
\frac{|T(z_1)-T(z_2)|^2}{|z_1-z_2|^2}=|T'(z_1)||T'(z_2)|
$$
for a M\"obius transformation $T$.

We will use a fact that each function $\varphi \in \mathcal B_p(\mathbb U)$ 
has boundary values almost everywhere on $\mathbb R$, 
and this boundary function $b(\varphi)$ belongs to the $p$-Besov space $B_p(\mathbb R)$.
As we have seen above, the results on the pairs $({\mathcal B}_p(\mathbb U),B_p(\mathbb R))$ 
and $({\mathcal B}_p(\mathbb D),B_p(\mathbb S))$ correspond under $T_*$;
we will consider this problem for $({\mathcal B}_p(\mathbb D),B_p(\mathbb S))$.

The boundary function $b(\phi)$ is given by the non-tangential limit of $\phi \in \mathcal B_p(\mathbb D)$.
The existence of the non-tangential limit, and moreover, the reproduction of $\phi$ from $b(\phi)$
by the Poisson integral have been proved (see \cite[Lemma 10.13]{Zh}).

For $p=2$, the fact that $b(\phi) \in B_2(\mathbb S)$ for $\phi \in \mathcal B_2(\mathbb D)$
is well-known as the Douglas formula for the Dirichlet integral:
$$
\Vert \phi \Vert_{\mathcal B_2}=\int_{\mathbb D} |\phi'(z)|^2 \frac{dxdy}{\pi}=
\int_{\mathbb S}\! \int_{\mathbb S} \frac{|b(\phi)(z)-b(\phi)(w)|^2}{|z-w|^2} \frac{|dz|}{2\pi}\frac{|dw|}{2\pi}
=\Vert b(\phi) \Vert_{B_2}.
$$
See \cite[Theorem 2-5]{Ah} for example.
The statement for the general case is as follows (see \cite[pp.131, 301]{Zh}). 

\begin{lemma}\label{131}
The boundary function $b(\phi)$ of $\phi \in \mathcal B_p(\mathbb D)$ belongs to $B_p(\mathbb S)$ for $p > 1$. 
The boundary extension operator $b:\mathcal B_p(\mathbb D) \to B_p(\mathbb S)$ is a bounded linear isomorphism onto the image.
\end{lemma}

A proof for the non-homogeneous Besov space on $\mathbb R$ (and on $\mathbb R^n$)
can be found in \cite[Section V.5]{St}.
A more explicit proof is in \cite[Theorems 2.1 and 5.1]{Pa}. These are referred to by \cite[p.505]{Rei}.

The inverse map of the boundary extension $b$ can be extended to $B_p(\mathbb S)$ as the Riesz--Szeg\"o projection.
The boundedness of this projection is also known (see \cite[Section 2.3]{Pe}).
This implies the boundedness of the conjugate operator.
On the real line $\mathbb R$, this is the {\it Hilbert transformation} $\mathcal H$, which is defined by
$$
({\mathcal H}u)(x)=\lim_{\varepsilon \to 0}\frac{1}{\pi} \int_{|x-t|>\varepsilon} \frac{u(t)}{x-t}dt
$$
(see \cite[Section III.1]{Ga}).

\begin{lemma}\label{Hilbert}
The Hilbert transformation $\mathcal H$ gives a bounded linear surjective isomorphism
$\mathcal H: B_p(\mathbb R) \to B_p(\mathbb R)$ such that $\mathcal H^2=-I$ for $p > 1$.
\end{lemma}

This also follows from the results on more general operators, for example, in \cite[Th\'eor\`eme A]{Lem} 
and \cite[Proposition 4.7]{GP}.

The $p$-Besov space $B_p(\mathbb R)$ is closely related to
BMO functions defined below.
A locally integrable complex-valued function $u$ on $\mathbb R$ is of {\it BMO} if
$$
\Vert u \Vert_{*} = \sup_{I \subset \mathbb R}\frac{1}{|I|} \int_I |u(x)-u_I| dx <\infty,
$$
where the supremum is taken over all bounded intervals $I$ on $\mathbb R$ and $u_I$ denotes the integral mean of $u$
over $I$. The set of all BMO functions on $\mathbb R$ is denoted by ${\rm BMO}(\mathbb R)$.
This is regarded as a Banach space with the BMO-norm $\Vert \cdot \Vert_*$
by ignoring the difference of constant functions. 
It is said that $u \in {\rm BMO}(\mathbb R)$ is of {\it VMO} if
$$ 
\lim_{|I| \to 0}\frac{1}{|I|} \int_I |u(x)-u_I| dx=0,
$$
and the set of all such functions is denoted by ${\rm VMO}(\mathbb R)$.

The following relation between $B_p(\mathbb R)$ and ${\rm VMO}(\mathbb R)$ is known.
See \cite[Section 3]{SW} and \cite[Propositions 2.2 and 2.3]{WM-3}.

\begin{proposition}\label{VMOBp}
$(1)$ If $u \in B_p(\mathbb R)$ then $u \in {\rm VMO}(\mathbb R)$. Moreover, $\Vert u \Vert _* \leq \Vert u \Vert_{B_p}$.
$(2)$ If $u \in B_p^{\mathbb R}(\mathbb R)$ then $e^u$ is an $A_\infty$-weight.
\end{proposition}

Here, we say that a non-negative 
locally integrable function $\omega \geq 0$ is an {\it $A_{\infty}$-weight} 
if there exists a constant $C_\infty(\omega) \geq 1$ such that
\begin{equation*}
\frac{1}{|I|} \int_I \omega(x) dx \leq C_\infty(\omega) \exp \left(\frac{1}{|I|} \int_I \log \omega(x) dx \right) 
\end{equation*}
for every bounded interval $I \subset \mathbb R$.
If $\omega$ is an $A_{\infty}$-weight, then $\log\omega$ is a BMO function (see \cite[Corollary IV.2.19]{GR}).  

Finally, we introduce classes of Beltrami coefficients on $\mathbb U$ including $\mathcal M_p(\mathbb U)$.
Let $\lambda$ be a positive Borel measure on the upper half-plane $\mathbb{U}$. 
We say that $\lambda$ is a {\em Carleson measure} if
$$
\Vert \lambda \Vert_c  = \sup_{I \subset \mathbb R} \frac{\lambda(I \times (0,|I|])}{|I|} < \infty,
$$
where the supremum is taken over all bounded closed interval $I \subset \mathbb R$ and
$I \times (0,|I|] \subset \mathbb U$ is the Carleson box. The set of all Carleson measures on $\mathbb U$ is denoted by
${\rm CM}(\mathbb U)$. 
A Carleson measure $\lambda \in {\rm CM}(\mathbb U)$ is called {\em vanishing} if
$$
\lim_{|I| \to 0}\frac{\lambda(I \times (0,|I|])}{|I|} = 0.
$$
The set of all vanishing Carleson measures on $\mathbb U$ is denoted by
${\rm CM}_0(\mathbb U)$.

For a Beltrami coefficient $\mu$ on $\mathbb{U}$, we define a positive Borel measure 
$\lambda_{\mu}$ that is absolutely continuous with respect to the Lebesgue measure and satisfies
$$
d\lambda_{\mu}(z) = |\mu(z)|^2y^{-1}dxdy. 
$$
Using this, a norm of $\mu$ is defined by $\Vert \mu \Vert_c=\Vert \lambda_\mu \Vert_c^{1/2}$.
Let ${\mathcal M}_c(\mathbb U)$ be the set of all Beltrami coefficients on $\mathbb U$ with $\Vert \mu \Vert_c<\infty$,
which is a domain of the Banach space with norm $\Vert \cdot \Vert_c +\Vert \cdot \Vert_\infty$.
The following claim implies the inclusion ${\mathcal M}_p(\mathbb U) \subset {\mathcal M}_c(\mathbb U)$.

\begin{proposition}\label{pBel}
If $\mu \in {\mathcal M}_p(\mathbb U)$ for $p \geq 1$, then $\lambda_\mu \in {\rm CM}_0(\mathbb U)$.
Moreover, for $p \geq 2$, $\Vert \mu \Vert_c \leq C_p \Vert \mu \Vert_p$ for some constant $C_p>0$ depending only on $p$.
\end{proposition}

\begin{proof}
For the first statement, we may assume that $p \geq 2$ because ${\mathcal M}_p(\mathbb U) \subset {\mathcal M}_2(\mathbb U)$
if $p < 2$.
Let $p'=p/2 \geq 1$ and take $q'\geq 1$ satisfying $1/p'+1/q'=1$. When $p'=1$, the inequality becomes simpler.
For any bounded interval $I \subset \mathbb R$, we have
\begin{align*}
&\quad \frac{1}{|I|}\int_0^{|I|}\! \int_I \frac{|\mu(z)|^2}{y}dxdy
=\frac{1}{|I|}\int_0^{|I|}\! \int_I \frac{|\mu(z)|^2}{y^{2/p'}}\cdot y^{\frac{2}{p'}-1}dxdy\\
&\leq \left(\int_0^{|I|}\! \int_I \frac{|\mu(z)|^p}{y^2}dxdy \right)^{2/p} \left(|I|^{1-q'}\int_0^{|I|} 
y^{-\left(1-\frac{1}{p'-1}\right)} dy \right)^{1/q'}.
\end{align*}
The first factor in the last line of the above inequality is bounded by $\Vert \mu \Vert_p^2$ and
tends to $0$ uniformly as $|I| \to 0$. The second factor is equal to $(p'-1)^{-1/q'}$ ($=1$ when $p'=1$), which we define as
$C_p^2$. Taking the square root shows the statement.
\end{proof}

\section{The Bers coordinates of the space of $p$-Weil--Petersson curves}

In this section, we introduce the Bers coordinates for the space of
$p$-Weil--Petersson embeddings $\mathbb R \to \mathbb C$ and show
the holomorphic correspondence to the $p$-Besov space
for $p>1$.

\begin{definition}
A continuous embedding $\gamma:\mathbb R \to \mathbb C$ passing through $\infty$ is called a
{\it $p$-Weil--Petersson embedding} for $p \geq 1$ if there is a quasiconformal homeomorphism 
$G:\mathbb C \to \mathbb C$ such that $G|_{\mathbb R}=\gamma$ and 
its complex dilatation $\mu_G=\bar \partial G/\partial G$ on $\mathbb U$ belongs to $\mathcal M_p(\mathbb U)$ and
$\mu_G$ on $\mathbb L$ belongs to $\mathcal M_p(\mathbb L)$.
We call such $G$ a {\it $p$-Weil--Petersson quasiconformal homeomorphism} associated with $\gamma$.
\end{definition}

The image of $\gamma:\mathbb R \to \mathbb C$ as above 
is called a {\it $p$-Weil--Petersson curve}. To consider the space of $p$-Weil--Petersson curves,
we add their parametrizations and regard them as $p$-Weil--Petersson embeddings. 
Special types of $p$-Weil--Petersson embeddings $\gamma$ are as follows. 
If such an embedding $\gamma$ maps $\mathbb R$ onto itself, this is nothing but a $p$-Weil--Petersson class homeomorphism. 
If $\gamma$ extends conformally to $\mathbb U$, we call
it the {\it Riemann mapping parametrization} of a $p$-Weil--Petersson curve.

We can define a {\it BMO embedding} $\gamma:\mathbb R \to \mathbb C$ by replacing the above $\mathcal M_p(\mathbb U)$
and $\mathcal M_p(\mathbb L)$ with $\mathcal M_c(\mathbb U)$ and $\mathcal M_c(\mathbb L)$.
By Proposition \ref{pBel}, we see that any $p$-Weil--Petersson embedding $\gamma$ for $p \geq 1$
is a BMO embedding.
Hence, we can utilize the following
known properties of BMO embeddings (see \cite[Proposition 3.3, Theorem 3.6]{WM-0})
for $p$-Weil--Petersson embeddings. 

\begin{proposition}\label{derivative}
A BMO-embedding $\gamma:\mathbb R \to \mathbb C$ has its derivative $\gamma'$ almost everywhere on $\mathbb R$ and $\log \gamma' \in {\rm BMO}(\mathbb R)$.
Moreover, if $|\gamma'|$ is an $A_\infty$-weight on $\mathbb R$, then 
$\gamma$ is locally absolutely continuous and the image $\gamma(\mathbb R)$ is a chord-arc curve.
\end{proposition}

A locally rectifiable Jordan curve $\Gamma$ 
passing through $\infty$ is called a {\it chord-arc curve} if there is
a constant $K \geq 1$ such that the length of the arc between $a, b \in \Gamma$
is bounded by $K|a-b|$.

\begin{lemma}\label{RM}
If $h:\mathbb R \to \mathbb C$ is a Riemann mapping parametrization of a $p$-Weil--Petersson curve for $p>1$, 
then $h$ is locally absolutely continuous with $\log h' \in B_p(\mathbb R)$.
\end{lemma}

\begin{proof}
By definition, $h$ extends to a conformal homeomorphism $H$ on $\mathbb U$. 
By Theorem \ref{Guo11} in Appendix, we have $\log H' \in \mathcal B_p(\mathbb U)$. Then,
the non-tangential limit $b(\log H')$ belongs to $B_p(\mathbb R)$ by Lemma \ref{131}.
On the other hand, 
$\log h'$ coincides with $b(\log H')$ (a.e.) since $H$ has the quasiconformal extension to $\mathbb C$
(see \cite[Theorem 5.5]{Pom}). Thus, we have $\log h' \in B_p(\mathbb R)$. In particular, $|h'| \in A_{\infty}$ by Proposition \ref{VMOBp}. Then, we conclude by Proposition \ref{derivative} that $h$ is locally absolutely continuous on $\mathbb R$. 
\end{proof}

The following result is given in \cite[p.1056]{Sh18} for $p = 2$, and in \cite[p.669]{TS} for $p \geq 2$. The generalization to $p > 1$ is possible as we also do in \cite[Theorem 5.5]{WM-1}.

\begin{lemma}\label{real-1}
If $f:\mathbb R \to \mathbb R$ is a
$p$-Weil--Petersson class homeomorphism for $p>1$, then $f$ is locally absolutely continuous with $\log f' \in B_p^{\mathbb R}(\mathbb R)$.
\end{lemma}

\begin{proof}
By the well-known conformal sewing principle
(see \cite[Section III.1.4]{Le} and \cite[p.11]{TT}), there exists a pair of quasiconformal homeomorphisms $H$ and $H_*$ on the whole plane $\mathbb C$ such that $H$ is conformal on $\mathbb U$, $H^*$ is conformal on $\mathbb L$, and 
$f =h_*^{-1} \circ h$ for $h=H|_{\mathbb R}$ and $h_*=H_*|_{\mathbb R}$ on $\mathbb R$.  
Moreover, we can choose these $H$ and $H_*$ so that the complex dilatation of $H|_{\mathbb L}$ is in $\mathcal M_p(\mathbb L)$
and the complex dilatation of $H_*|_{\mathbb U}$ is in $\mathcal M_p(\mathbb U)$.
This is a crucial step and its argument is in \cite{WM-1}.
We note that to obtain the appropriate mapping $H_*$, we have to show that the inverse $f^{-1}$ has
a quasiconformal extension whose complex dilatation belongs to $\mathcal M_p(\mathbb U)$.
This is a part of the property that $T_p$ has the group structure, which will be explained in Section 6.

By Lemma \ref{RM}, we have $h$ and $h_*$ are locally absolutely continuous with 
$\log h' \in B_p(\mathbb R)$ and $\log h_*' \in B_p(\mathbb R)$.
From $h_* \circ f =h$ on $\mathbb R$, we see that the increasing homeomorphism
$f$ maps a set of null measure to a set of null measure
because $h$ is locally absolutely continuous and $|h_*'(x)| > 0$ almost everywhere on $\mathbb R$.
Hence, $f$ is locally absolutely continuous.
Taking the derivatives of the both sides of the above equality, we have
\begin{equation*}\label{sewing}
P_f(\log h_*')+\log f' =\log h'.
\end{equation*}
By Proposition \ref{pullback}, we see that $P_f(\log h_*') \in B_p(\mathbb R)$. Hence,
$\log f' \in B_p(\mathbb R)$. 
\end{proof}

The combination of the above two lemmas yields the general property on $p$-Weil--Petersson embeddings.

\begin{theorem}\label{curve}
A $p$-Weil--Petersson embedding 
$\gamma:\mathbb R \to \mathbb C$ is locally absolutely continuous and $\log \gamma'$
belongs to $B_p(\mathbb R)$ for $p>1$. 
\end{theorem}

\begin{proof}
Let $G:\mathbb C \to \mathbb C$ be a $p$-Weil--Petersson quasiconformal homeomorphism 
associated with $\gamma$ such that
$\mu_1=\mu_G|_{\mathbb U} \in \mathcal M_p(\mathbb U)$ and $\mu_2=\mu_G|_{\mathbb L} \in \mathcal M_p(\mathbb L)$.
We take a quasiconformal homeomorphism $F:\mathbb C \to \mathbb C$ whose complex dilatation is
$\mu_1(z)$ for $z \in \mathbb U$ and $\overline{\mu_1(\bar z)}$ for $z \in \mathbb L$, 
which maps $\mathbb R$ onto itself.
By Lemma \ref{real-1}, $f=F|_{\mathbb R}$ is locally absolutely continuous and $\log f'$ belongs to $B_p^{\mathbb R}(\mathbb R)$.
Next, we take a quasiconformal homeomorphism $H:\mathbb C \to \mathbb C$ that is conformal on $\mathbb U$ and
whose complex dilatation on $\mathbb L$ is the push-forward $F_* \mu_2$ of $\mu_2$ by $F$. Namely, 
the complex dilatation of $H \circ F$ is $\mu_2$. Then, $H \circ F$ coincides with $G$ up to
an affine transformation of $\mathbb C$, and hence, we may assume that $H \circ F=G$.

We may replace $F|_{\mathbb L}$ with 
a bi-Lipschitz diffeomorphism 
under the hyperbolic metric 
whose complex dilatation $\tilde \mu_1$ belongs to $\mathcal M_p(\mathbb L)$ 
(see \cite[Theorem 6]{Cu}, \cite[Theorem 2.4]{Ya}, and \cite[Lemma 3.4]{WM-1}).
The complex dilatation of $H|_{\mathbb L}$ is explicitly given by
$$
F_* \mu_2(\zeta)=\frac{\mu_2(z)-\tilde \mu_1(z)}{1-\overline{\tilde \mu_1(z)}\mu_2(z)}\cdot \frac{F_z}{\overline{ F_z}} 
$$
for $\zeta=F(z) \in \mathbb L$. Using the property that $F$ is a bi-Lipschitz diffeomorphism, we see 
from this formula that
$F_* \mu_2$ also belongs to $\mathcal M_p(\mathbb L)$.
Then by Lemma \ref{RM}, $h=H|_{\mathbb R}$ is locally absolutely continuous with $\log h' \in B_p(\mathbb R)$. 

By $h \circ f=\gamma$, we see that $\gamma$ is also locally absolutely continuous, 
and taking the derivative,
we have 
$$
\log h' \circ f + \log f'=\log \gamma'.
$$
By $\log f' \in B_p(\mathbb R)$ and $\log h' \in B_p(\mathbb R)$ combined with Proposition \ref{pullback},
we obtain that $\log \gamma' \in B_p(\mathbb R)$. 
\end{proof}

We impose the normalization $\gamma(0)=0$ and $\gamma(1)=1$ (and $\gamma(\infty)=\infty$) on a
$p$-Weil--Petersson embedding $\gamma$. Let ${\rm WPC}_p$ be the set of all such normalized $p$-Weil--Petersson embeddings for $p \geq 1$.
We also denote the subset of ${\rm WPC}_p$ consisting of all normalized 
$p$-Weil--Petersson class homeomorphisms by $W_p$, and the subset consisting of all normalized 
Riemann mapping parametrizations of $p$-Weil--Petersson curves by $RM_p$.
For $\mu_1 \in \mathcal M_p(\mathbb U)$ and $\mu_2 \in \mathcal M_p(\mathbb L)$,
we denote by $G=G(\mu_1,\mu_2)$ the normalized $p$-Weil--Petersson quasiconformal homeomorphism of $\mathbb C$ 
($G(0)=0$, $G(1)=1$, and $G(\infty)=\infty$) with $\mu_G|_{\mathbb U}=\mu_1$ and $\mu_G|_{\mathbb L}=\mu_2$.
We define a map
$$
\widetilde \iota:\mathcal M_p(\mathbb U) \times \mathcal M_p(\mathbb L) \to {\rm WPC}_p
$$
by $\widetilde \iota(\mu_1,\mu_2)=G(\mu_1,\mu_2)|_{\mathbb R}$.
Then, by the famous argument of simultaneous uniformization due to Bers,
we see the following fact.

\begin{proposition}\label{id}
The space ${\rm WPC}_p$ of all normalized $p$-Weil--Petersson embeddings is identified with
$T_p(\mathbb U) \times T_p(\mathbb L)$ for $p \geq 1$. More precisely, $\widetilde \iota$ splits into a well-defined bijection
$$
\iota:T_p(\mathbb U) \times T_p(\mathbb L) \to {\rm WPC}_p
$$
by the product of the Teichm\"uller projections
$
\widetilde \pi:\mathcal M_p(\mathbb U) \times \mathcal M_p(\mathbb L) \to T_p(\mathbb U) \times T_p(\mathbb L)
$
such that $\widetilde \iota=\iota \circ \widetilde \pi$.
\end{proposition}

We call $T_p(\mathbb U) \times T_p(\mathbb L)$ the {\it Bers coordinates} of ${\rm WPC}_p$. Any
normalized $p$-Weil--Petersson embedding $\gamma$ is represented by a pair $([\mu_1],[\mu_2])$ of
the Teichm\"uller equivalence classes of $\mu_1 \in \mathcal M(\mathbb U)$ and $\mu_2 \in \mathcal M(\mathbb L)$
via $G(\mu_1,\mu_2)$.

We may provide complex Banach manifold structures for $T_p(\mathbb U)$ and $T_p(\mathbb L)$ by
using the pre-Schwarzian derivative models as in Theorem \ref{model} of Appendix. 
Namely, $T_p(\mathbb U)$ is identified with the domain
$\mathcal T_p(\mathbb L)$ 
of the analytic $p$-Besov space ${\mathcal B}_p(\mathbb L)$, and $T_p(\mathbb L)$ is identified with the domain $\mathcal T_p(\mathbb U)$
of ${\mathcal B}_p(\mathbb U)$ for $p>1$:
\begin{align*}
T_p(\mathbb U) &\cong \mathcal T_p(\mathbb L)=\{ {\mathcal L}_{G(\mu,0)} \in {\mathcal B}_p(\mathbb L) \mid \mu \in \mathcal M_p(\mathbb U)\};\\
T_p(\mathbb L) &\cong \mathcal T_p(\mathbb U)=\{ {\mathcal L}_{G(0,\mu)} \in {\mathcal B}_p(\mathbb U) \mid \mu \in \mathcal M_p(\mathbb L)\}.
\end{align*}
Then, by 
Proposition \ref{id}, we may also regard ${\rm WPC}_p$ as a domain 
of ${\mathcal B}_p(\mathbb L) \times {\mathcal B}_p(\mathbb U)$ for $p>1$.

By Theorem \ref{curve}, we can consider an injective map $L:{\rm WPC}_p \to B_p(\mathbb R)$ defined 
by $L(\gamma)=\log \gamma'$. Then, with respect to the complex structure of ${\rm WPC}_p$ given as above,
we see the following.

\begin{theorem}\label{holo}
The map $L:{\rm WPC}_p \to B_p(\mathbb R)$ is a holomorphic injection for $p>1$.
\end{theorem}

\begin{proof}
We will prove that $L$ is holomorphic at any point $\gamma=G(\mu_1,\mu_2)|_{\mathbb R}$ in ${\rm WPC}_p$.
Since ${\rm WPC}_p$ can be regarded as a domain of the product 
${\mathcal B}_p(\mathbb L) \times {\mathcal B}_p(\mathbb U)$ of the Banach spaces, 
the Hartogs theorem for Banach spaces (see \cite[Theorem 14.27]{Ch} and \cite[Theorem 36.8]{Mu}) implies that
we have only to prove that $L$ is separately holomorphic. Thus, 
by fixing $[\mu_1] \in T_p(\mathbb U)$, we will show that 
$\log (G(\mu_1,\mu)|_{\mathbb R})' \in B_p(\mathbb R)$ depends holomorphically on $[\mu] \in T_p(\mathbb L)$.
The other case is similarly treated.

By the proof of Theorem \ref{curve}, we have
$$
\log (G(\mu_1,\mu)|_{\mathbb R})' =\log \gamma'=\log h' \circ f + \log f'.
$$
We may choose a bi-Lipschitz diffeomorphism $F:\mathbb L \to \mathbb L$
that is the extension of 
$f: \mathbb R \to \mathbb R$  and
whose complex dilatation 
still belongs to $\mathcal M_p(\mathbb L)$ as before (see \cite[Lemma 3.4]{WM-1}).
Let
$h:\mathbb R \to \mathbb C$ be the restriction of
the quasiconformal homeomorphism $H_{F_*\mu}$ of $\mathbb C$ that is conformal on $\mathbb U$ and
has the complex dilatation $F_* \mu$ on $\mathbb L$. 
Since $F$ is a
bi-Lipschitz diffeomorphism, we see that $F_*$ acts on ${\mathcal M}_p(\mathbb L)$
as a biholomorphic automorphism, and its action projects down to $T_p(\mathbb L)$ 
also as a biholomorphic automorphism. For $p \geq 2$, this is shown in \cite[Chap.1, Corollary 2.12]{TT} and \cite[Proposition 5.3]{Ya},
and the same proof is valid for $p \geq 1$ once we know that $F_* \mu \in {\mathcal M}_p(\mathbb L)$ for every $\mu \in {\mathcal M}_p(\mathbb L)$
(see \cite[Lemma 3.1]{WM-1}). 
Continuity or local boundedness of $F_*$ is enough to show the holomorphy of $F_*$, which is also explained in
\cite[Proposition 3.1]{WM-2} in a similar setting.
Hence, ${\mathcal L}_{H_{F_*\mu}|_{\mathbb U}}=\log (H_{F_*\mu}|_{\mathbb U})'
\in {\mathcal B}_p(\mathbb U)$ depends on $[\mu] \in T_p(\mathbb L)$ holomorphically
as we see that $\alpha^{-1}: \mathscr T_p \cong T_p(\mathbb L) \to {\mathcal B}_p(\mathbb U)$ is holomorphic in the proof of Theorem \ref{model} in Appendix.

By Lemma \ref{131}, we see that the boundary extension $b:{\mathcal B}_p(\mathbb U) \to B_p(\mathbb R)$ is
a bounded linear operator for $p>1$. Moreover, by Proposition \ref{pullback}, the variable change operator
$P_f:B_p(\mathbb R) \to B_p(\mathbb R)$ induced by $f$ is also a bounded linear operator.
Therefore,
$$
\log h' \circ f=P_f \circ b ({\mathcal L}_{H_{F_*\mu}|_{\mathbb U}}) \in B_p(\mathbb R)
$$
in particular depends on $[\mu] \in T_p(\mathbb L)$ holomorphically, and so does $\log (G(\mu_1,\mu)|_{\mathbb R})'$.
\end{proof}

\section{Conformal welding and curve theoretic coordinates}

We introduce other coordinates of ${\rm WPC}_p$ and $L({\rm WPC}_p)$, and investigate their
relationship. To this end, we utilize the canonical automorphisms of ${\rm WPC}_p$. 

For $\nu \in {\mathcal M}_p(\mathbb U)$,
the same symbol $\nu$ still denotes the complex dilatation $\overline{\nu(\bar z)}$ for $z \in \mathbb L$ in 
${\mathcal M}_p(\mathbb L)$. This also gives the identification of $T_p(\mathbb U)$ and $T_p(\mathbb L)$,
which is often denoted by $T_p$ hereafter.
For any $[\nu] \in T_p$, we define the {\it right translation} of ${\rm WPC}_p$ for $p \geq 1$
by
$$
\widetilde R_{[\nu]}:([\mu_1],[\mu_2]) \mapsto ([\mu_1] \ast [\nu], [\mu_2] \ast [\nu]),
$$
where $R_{[\nu]}([\mu])=[\mu] \ast [\nu]$ is the composition of elements in $T_p$ that is given by the Teichm\"uller class
of the complex dilatation of $F^{\mu} \circ F^{\nu}$ for the normalized 
$p$-Weil--Petersson class homeomorphisms $F^{\mu}$ and $F^{\nu}$ of 
$\mathbb U$ (or $\mathbb L$) onto itself with the given complex dilatations.
It is known that the right translation $R_{[\nu]}$ is a biholomorphic automorphism of $T_p$ for $p \geq 1$
as is seen in the proof of Theorem \ref{holo}.
Hence, $\widetilde R_{[\nu]}$ gives a biholomorphic automorphism of ${\rm WPC}_p$.

First, we consider the {\it conformal welding coordinates} of ${\rm WPC}_p$ for $p \geq 1$.
Under the Bers coordinates ${\rm WPC}_p \cong T_p(\mathbb U) \times T_p(\mathbb L)$,
the subspace $W_p \subset {\rm WPC}_p$ is identified with the diagonal locus
$$
\{([\mu],[\mu]) \in T_p(\mathbb U) \times T_p(\mathbb L) \mid [\mu] \in T_p\}.
$$
Since this is the fixed point locus of the anti-holomorphic involution 
$([\mu_1],[\mu_2]) \mapsto ([\mu_2],[\mu_1])$, 
we see that $W_p$ is a real-analytic submanifold of ${\rm WPC}_p$. Moreover, the subspace $RM_p \subset {\rm WPC}_p$ 
is identified with the second coordinate axis
$$
\{([0],[\mu]) \in T_p(\mathbb U) \times T_p(\mathbb L) \mid [\mu] \in T_p\},
$$
which is a complex-analytic submanifold of ${\rm WPC}_p$. 

We define the projections to these submanifolds
$$
\Pi: {\rm WPC}_p \to W_p, \qquad \Phi: {\rm WPC}_p \to RM_p
$$
by $\Pi([\mu_1],[\mu_2])=([\mu_1],[\mu_1])$ and $\Phi([\mu_1],[\mu_2])=([0],[\mu_2] \ast [\mu_1]^{-1})$ in the Bers coordinates, where $[\mu]^{-1}$ is the inverse of an element in $T_p$ that is given by
the Teichm\"uller class of the complex dilatation $\mu^{-1}$ of $(F^\mu)^{-1}$. 
Then, every $\gamma \in {\rm WPC}_p$ 
is decomposed uniquely into $\gamma=\Phi(\gamma) \circ \Pi(\gamma)$.
This corresponds to the decomposition $\gamma=h \circ f$ in the proof of Theorem \ref{curve}.
Clearly, $\Pi$ is real-analytic. 
We see that $\Phi$ is continuous later by Theorem \ref{topgroup}. 
The biholomorphic automorphism $\widetilde R_{[\nu]}$ of ${\rm WPC}_p$ for $[\nu] \in T_p$
satisfies that $\Phi \circ \widetilde R_{[\nu]}=\Phi$.

The projections $\Pi$ and $\Phi$ defines another product structure $W_p \times RM_p$ on ${\rm WPC}_p$ for $p \geq 1$.
Namely, we have a bijection
$$
(\Pi,\Phi):{\rm WPC}_p \to W_p \times RM_p.
$$
Once we see that $\Phi$ is continuous, $(\Pi,\Phi)$ is a homeomorphism.
This is the
coordinate change of ${\rm WPC}_p$ from the Bers coordinates
to the one we may call the
conformal welding coordinates. Since $W_p$ and $RM_p$ are both identified with $T_p$, 
by marking $T_p$ with $W_p \cong T_p^W$ and $RM_p \cong T_p^{RM}$,
the coordinate change is expressed as
$$
T_p(\mathbb U) \times T_p(\mathbb L) \to T_p^W \times T_p^{RM}: \quad([\mu_1], [\mu_2]) \mapsto ([\mu_1],[\mu_2] \ast [\mu_1]^{-1}).
$$

Next, we consider the {\it curve theoretical coordinates} of the space of $p$-Weil--Petersson embeddings by using 
the image $L({\rm WPC}_p)$ in $B_p(\mathbb R)$ for $p>1$. 
We see that $L$ is injective because $\gamma$ can be reproduced from $w=\log \gamma' \in L({\rm WPC}_p)$ by
$$
\gamma(x)=\int_0^x e^{w(t)}dt.
$$
We have assumed that $B_p(\mathbb R)$ is the Banach space of all equivalence classes of 
complex-valued functions modulo additive constants,
and can also regard it as the set of representatives $w$ satisfying the normalization condition
$\int_0^1 e^{w(t)}dt=1$. For $w \in L({\rm WPC}_p)$, this is always possible by adding some complex constant to $w$.

Let $u \in B_p^{\mathbb R}(\mathbb R)$ and let
$\gamma_{u}:\mathbb R \to \mathbb R$ be the $p$-Weil--Petersson class homeo\-morphism in $W_p$
defined by $\gamma_{u}(x)=\int_0^x e^{u(t)}dt$. Then, the variable change operator
$P_{\gamma_u}:B_p(\mathbb R) \to B_p(\mathbb R)$ is given by $w \mapsto w \circ \gamma_{u}$ for $w \in B_p(\mathbb R)$,
which is a linear isomorphism of the Banach space $B_p(\mathbb R)$ onto itself by Proposition \ref{pullback}. 
Moreover, we define $Q_{u}(w)=P_{\gamma_u}(w)+u$ for $w \in B_p(\mathbb R)$, 
which is an affine isomorphism of $B_p(\mathbb R)$ onto itself. Since $P_{\gamma_u}$ preserves $B_p^{\mathbb R}(\mathbb R)$,
we see that $Q_u$ maps $u_0+iB_p^{\mathbb R}(\mathbb R)$ onto $u_1+iB_p^{\mathbb R}(\mathbb R)$
for some $u_1 \in B_p^{\mathbb R}(\mathbb R)$ depending on $u_0 \in B_p^{\mathbb R}(\mathbb R)$,
where $iB_p^{\mathbb R}(\mathbb R)$ denotes the real subspace of $B_p(\mathbb R)$ consisting of all purely imaginary functions
modulo complex-valued constant functions. 

We see that the affine isomorphism $Q_u$ of $B_p(\mathbb R)$ keeps the subset $L({\rm WPC}_p)$ invariant.
We often use the correspondence between $u \in B_p^{\mathbb R}(\mathbb R)$ and $[\nu] \in T_p$
through $\gamma_u \in W_p \cong T_p$.

\begin{proposition}\label{Lequation}
The right translation $\widetilde R_{[\nu]}$ satisfies  
$$L \circ \widetilde R_{[\nu]}=Q_{L([\nu],[\nu])} \circ L$$
on ${\rm WPC}_p$ for every $[\nu] \in T_p$. Hence, the affine isomorphism $Q_u$ of $B_p(\mathbb R)$
for any $u \in B_p^{\mathbb R}(\mathbb R)$ maps $L({\rm WPC}_p)$ onto itself.
\end{proposition}

\begin{proof}
For any $([\mu_1],[\mu_2]) \in {\rm WPC}_p$, we have
\begin{equation*}
\begin{split}
L \circ \widetilde R_{[\nu]}([\mu_1],[\mu_2])&=L([\mu_1] \ast [\nu], [\mu_2] \ast [\nu])\\
&=\log \gamma'_{L([\mu_1] \ast [\nu], [\mu_2] \ast [\nu])}=\log (\gamma_{L([\mu_1],[\mu_2])}\circ \gamma_{L([\nu],[\nu])})'\\
&=\log \gamma'_{L([\mu_1],[\mu_2])}\circ \gamma_{L([\nu],[\nu])}+ \log \gamma'_{L([\nu],[\nu])}\\
&=L([\mu_1],[\mu_2]) \circ \gamma_{L([\nu],[\nu])}+L([\nu],[\nu])=Q_{L([\nu],[\nu])} \circ L([\mu_1],[\mu_2])
\end{split}
\end{equation*}
as required.
For any $u \in B_p^{\mathbb R}(\mathbb R)$, we choose $[\nu] \in T_p$ corresponding to 
$\gamma_u \in W_p$.
Namely, $L([\nu],[\nu])=u$. Then,
we have
$$
Q_u(L({\rm WPC}_p))=Q_{L([\nu],[\nu])} \circ L({\rm WPC}_p)=L \circ \widetilde R_{[\nu]}({\rm WPC}_p)=L({\rm WPC}_p),
$$
which proves the second statement.
\end{proof}

We define the following subset of $L({\rm WPC}_p)$ corresponding to the {\it arc-length parametrization}:
$$
iB_p^{\mathbb R}(\mathbb R) ^{\circ}=iB_p^{\mathbb R}(\mathbb R) \cap L({\rm WPC}_p).
$$
Let $iv \in iB_p^{\mathbb R}(\mathbb R) ^{\circ}$.
Then,
$\gamma_{iv}(x)=\int_0^x e^{iv(t)}dt$ is a $p$-Weil--Petersson embedding of arc-length parametrization.
Precisely speaking, by the normalization, $\gamma_{iv}$ is parametrized by
the multiple of its arc-length by a positive constant.
We can regard $iB_p^{\mathbb R}(\mathbb R) ^{\circ}$ as a parameter space of such $p$-Weil--Petersson embeddings for $p>1$.
All $p$-Weil--Petersson embeddings are obtained by the reparametrization of
their arc-length parametrizations as follows.

\begin{lemma}\label{arclength}
Let $u \in B_p^{\mathbb R}(\mathbb R)$ and $iv \in iB_p^{\mathbb R}(\mathbb R)^{\circ}$.
Then, $\gamma_{Q_u(iv)}(x)$ is obtained from the $p$-Weil--Petersson embedding $\gamma_{iv}(x')$ of arc-length parametrization
by the change of parameter $x'= \gamma_u(x)$, which is also 
a $p$-Weil--Petersson embedding. Conversely, every $p$-Weil--Petersson embedding is obtained in this way.
Hence, the map 
$$
J:B_p^{\mathbb R}(\mathbb R) \times iB_p^{\mathbb R}(\mathbb R)^{\circ} \to L({\rm WPC}_p) \subset B_p(\mathbb R)
$$
defined by $J(u,iv)=Q_u(iv)=u+iP_{\gamma_u}(v)$ is bijective.
\end{lemma}

\begin{proof}
Since $Q_u(iv)=u+iP_{\gamma_u}(v)=u+iv \circ \gamma_u$, we have
\begin{equation*}
\begin{split}
\gamma_{Q_u(iv)}(x)=\int_0^x e^{u(t)} e^{iv \circ \gamma_u(t)}dt
=\int_0^x \gamma'_u(t) e^{iv \circ \gamma_u(t)}dt
=\int_0^{\gamma_u(x)} e^{iv(s)}ds=\gamma_{iv}(\gamma_u(x))
\end{split}
\end{equation*}
by $s=\gamma_u(t)$. Proposition \ref{Lequation} implies that the reparametrization of a $p$-Weil--Petersson embedding is
also a $p$-Weil--Petersson embedding. Conversely,
let $\gamma_{u+iv'}$ be any $p$-Weil--Petersson embedding for $u+iv' \in L({\rm WPC}_p)$. 
Then, by choosing 
$v \in B_p^{\mathbb R}(\mathbb R)$ satisfying $P_{\gamma_u}(v)=v'$, we see that $\gamma_{u+iv'}$ is obtained
from $\gamma_{iv}$ by the change of the parameter. Then, $\gamma_{iv}$ is a $p$-Weil--Petersson embedding,
and hence $iv \in iB_p^{\mathbb R}(\mathbb R)^{\circ}$.
\end{proof}

We will see later in Theorem \ref{Jhomeo} that the above bijection $J$ is in fact a homeomorphism.

Now, we have two product structures $W_p \times RM_p$ and $B_p^{\mathbb R}(\mathbb R) \times iB_p^{\mathbb R}(\mathbb R)^{\circ}$
on ${\rm WPC}_p$ for $p>1$. There is a close relation between these structures through $L$ and $J$.
Each fiber of the projection $\Phi$ consists of a family of embeddings with the same image, 
and hence their arc-length parametrizations are
the same. This observation leads the following.

\begin{proposition}\label{fibers}
For any $iv \in iB_p^{\mathbb R}(\mathbb R)^{\circ}$, let 
$\gamma=\Phi \circ L^{-1} \circ J(0,iv) \in RM_p$. 
Then, the fiber $\Phi^{-1}(\gamma)$ coincides with $L^{-1}\circ J(B_p^{\mathbb R}(\mathbb R) \times \{iv\})$, which is the family of normalized
$p$-Weil--Petersson embeddings with the same image $\gamma(\mathbb R)$.
\end{proposition}

\section{Biholomorphic correspondence}

{\em All the results in this section are stated under the assumption $p>1$}.
We prove the main theorem in this section as follows.

\begin{theorem}\label{biholo}
The holomorphic map $L:{\rm WPC}_p \to B_p(\mathbb R)$ is a biholomorphic homeomorphism onto 
its image. In particular, $L({\rm WPC}_p)$ is an open contractible domain of $B_p(\mathbb R)$
which contains $B_p^{\mathbb R}(\mathbb R)$.
\end{theorem}

\begin{proof}
In virtue of Theorem \ref{holo}, 
in order to prove that $L$ is biholomorphic,
it suffices to show that $L$ has a local holomorphic inverse at 
any point $w \in L({\rm WPC}_p) \subset B_p(\mathbb R)$.
This in particular shows that $L({\rm WPC}_p)$ is open.

It is proved in \cite[Theorem 6.1]{SW} based on the arguments in \cite{Se}
that if $w=iv \in iB_p^{\mathbb R}(\mathbb R)^{\circ}$ for $p=2$,
then there is a neighborhood $V_{iv} \subset L({\rm WPC}_p)$ of $iv$ and
a holomorphic map $\lambda_{iv}:V_{iv} \to \mathcal M_p(\mathbb U) \times \mathcal M_p(\mathbb L)$ 
such that $L \circ \widetilde \pi \circ \lambda_{iv}={\rm id}|_{V_{iv}}$.
For general $p >1$, the proof is essentially the same.

By \cite[Lemma 4.11]{Se} and \cite[Proposition 5.3]{SW}, the quasiconformal homeomorphism $G$ of $\mathbb C$ onto iteself defined by 
the complex dilatation $\lambda_{iv}(iv) \in  \mathcal M_p(\mathbb U) \times \mathcal M_p(\mathbb L)$
is bi-Lipschitz in the Euclidean metric on $\mathbb C$. This implies that both $G|_{\mathbb U}$ and $G|_{\mathbb L}$
are bi-Lipschitz with respect to the hyperbolic metrics on $\mathbb U$, $\mathbb L$, and their images.
In this case, the product of the Teichm\"uller projections 
$\widetilde \pi:\mathcal M_p(\mathbb U) \times \mathcal M_p(\mathbb L) \to T(\mathbb U) \times T(\mathbb L)$
is continuous at $\lambda_{iv}(iv)$, and in fact differentiable.
This is shown in \cite[Lemma 3.2]{WM-1} for $p \geq 1$, although the Teichm\"uller projection is known to be holomorphic
for $p \geq 2$ in \cite[Theorem 3.1]{Tang}. Let $\Psi_{iv}=\widetilde \pi \circ \lambda_{iv}$ and $\gamma=\Psi_{iv}(iv) \in {\rm WPC}_p$.
By $L \circ \Psi_{iv}={\rm id}|_{V_{iv}}$, their derivatives satisfy $d_{\gamma}L \circ d_{iv} \Psi_{iv}={\rm id}$.
Then, the inverse function theorem (see \cite[Theorem 7.18]{Ch}) asserts that $L$ has a local holomorphic inverse on $V_{iv}$ by choosing the domain smaller if necessary. 

If $w=u+iv'$ is an arbitrary point in $L({\rm WPC}_p)$, then we utilize 
$Q_u$, and find $iv \in iB_p^{\mathbb R}(\mathbb R)^\circ$ with $Q_u(iv)=u+iv'$ by Lemma \ref{arclength}.
Since $Q_u$ is a biholomorphic automorphism of $L({\rm WPC}_p)$ by Proposition \ref{Lequation},
we see that $\widetilde R_{[\nu]}\circ \Psi_{iv} \circ Q_u^{-1}$ is holomorphic on $Q_u(V_{iv})$
for $[\nu] \in T_p$ corresponding to $\gamma_u \in W_p$. Then, Proposition \ref{Lequation} implies that
$$
L \circ \widetilde R_{[\nu]} \circ \Psi_{iv} \circ Q_u^{-1}=Q_{L([\nu],[\nu])} \circ L \circ \Psi_{iv} \circ Q_u^{-1}
=Q_{L([\nu],[\nu])} \circ Q_u^{-1}=\rm id
$$
on $Q_u(V_{iv})$. Hence, $\widetilde R_{[\nu]}\circ \Psi_{iv} \circ Q_u^{-1}$ is 
a local holomorphic inverse of $L$ on $Q_u(V_{iv})$.

We know that the $p$-Weil--Petersson Teichm\"uller space $T_p$ is contractible by \cite[Theorem 6]{Cu} for $p=2$, by \cite[Proposition 3.5]{Ya}
for $p \geq 2$, and by the next Corollary \ref{real-analytic} (1) implying the homeomorphic identification $T_p \cong W_p \cong B_p^{\mathbb R}(\mathbb R)$
for general $p >1$. Hence, the product $T_p(\mathbb U) \times T_p(\mathbb L)$
is contractible, and so is $L({\rm WPC}_p)$.
\end{proof}

\begin{remark}
The space $L({\rm WPC}_p)$ is denoted by 
$\widehat{\mathcal T}_e$ in \cite[Theorem 2.5]{SW} in the case of $p=2$ and proved that it is
a contractible open domain in $H^{1/2}(\mathbb R)=B_2(\mathbb R)$.
It is also shown in \cite[Theorem 2.2]{SW} that $iB_2^{\mathbb R}(\mathbb R)^{\circ}$,
the parameter space for Weil--Petersson curves with arc-length parametrization, 
coincides with an open subset of
$iB_2^{\mathbb R}(\mathbb R)$ consisting of all elements corresponding to chord-arc curves with
arc-length parametrization. This result can be generalized to any $p >1$.
\end{remark}

Let $IW_p \subset {\rm WPC}_p$ denote the subset of all arc-length parametrizations of normalized 
$p$-Weil--Petersson embeddings. Namely,
$$
IW_p=L^{-1}(iB_p^{\mathbb R}(\mathbb R)^{\circ}).
$$
As $iB_p^{\mathbb R}(\mathbb R)^{\circ}$ is a real-analytic submanifold of 
the domain $L({\rm WPC}_p)$ in the complex Banach space $B_p(\mathbb R)$,
$IW_p$ is a real-analytic submanifold of the complex manifold ${\rm WPC}_p$.
By $W_p=L^{-1}(B_p^{\mathbb R}(\mathbb R))$, we again see that
$W_p$ is a real-analytic submanifold of ${\rm WPC}_p$.

\begin{corollary}\label{real-analytic}
$(1)$
$W_p$ is a real-analytic submanifold of ${\rm WPC}_p$ that is
the diagonal of $T_p(\mathbb U) \times T_p(\mathbb L)$, and
$L|_{W_p}$ 
is a real-analytic homeomorphism onto $B_p^{\mathbb R}(\mathbb R)$ whose inverse 
is also real-analytic.
$(2)$
$IW_p$ is a real-analytic submanifold of ${\rm WPC}_p$, and
$L|_{IW_p}$ 
is a real-analytic homeomorphism onto $iB_p^{\mathbb R}(\mathbb R)^\circ$ whose inverse 
is also real-analytic.
\end{corollary}

The real analytic property of $L|_{W_p}$ has been shown in \cite[Theorem 2.3]{ST} in the case of $p=2$
by a different method. 
Part (1) of the above corollary shows that $W_p$ is equipped with both the complex-analytic structure of $T_p$
and the real-analytic structure of $B_p^{\mathbb R}(\mathbb R)$, which are real-analytically equivalent.
Later in Theorem \ref{IWp}, we will see that $IW_p$ is canonically equipped with the complex-analytic structure of $T_p$
which is topologically equivalent to the real-analytic structure of $iB_p^{\mathbb R}(\mathbb R)^\circ$.

In \cite[Theorem 4.4]{WM-3}, we construct a holomorphic map $\Lambda:U(B^{\mathbb R}_p(\mathbb R)) \to \mathcal M_p(\mathbb U)$
on some neighborhood $U(B^{\mathbb R}_p(\mathbb R))$ of the real-valued subspace $B^{\mathbb R}_p(\mathbb R)$ for $p>1$.
In the same way, we have the correspondence to the complex dilatations on $\mathbb L$. Thus,
we can extend $\Lambda$ to a holomorphic map
$$
\widetilde \Lambda:U(B^{\mathbb R}_p(\mathbb R)) \to \mathcal M_p(\mathbb U) \times \mathcal M_p(\mathbb L).
$$
This induces the inverse of $L$ on the neighborhood $U(B^{\mathbb R}_p(\mathbb R))$ as is shown in \cite[Theorem 4.5]{WM-3}.

\begin{theorem}\label{section}
The neighborhood $U(B^{\mathbb R}_p(\mathbb R))$ is contained in $L({\rm WPC}_p)$,
and 
$$\widetilde \pi \circ \widetilde \Lambda:U(B^{\mathbb R}_p(\mathbb R)) \to {\rm WPC}_p$$
is the inverse of $L$ on $U(B^{\mathbb R}_p(\mathbb R))$ which is holomorphic. 
\end{theorem}

We compare the arc-length parametrizations $IW_p$ with the Riemann mapping para\-metri\-zations $RM_p$. 
Both are the sets of all representatives of $p$-Weil--Petersson curves,
which follows from Proposition \ref{fibers}.
Hence, there is a canonical bijection between $IW_p$ and $RM_p$ giving the change of the representatives,
namely, keeping the images of the corresponding embeddings the same. 
For the projection $\Phi:{\rm WPC}_p \to RM_p$,
this bijection is nothing but its restriction $\Phi|_{IW_p}:IW_p\to RM_p$.
We will see that $\Phi |_{IW_p}$ is a homeomorphism by Proposition \ref{continuity} in the next section. 

Here, we consider the other projection $\Pi$ restricted to $IW_p$,
which has been studied with great interest in the literature.
For any $\gamma \in IW_p$, $\Pi(\gamma) \in W_p$ is defined by
the $p$-Weil--Petersson class homeomorphism inducing
the parameter change from $\gamma$ to $\Phi(\gamma) \in RM_P$. 
We will prove the bi-real-analytic property of this mapping.
For the space of chord-arc curves, this property for the corresponding map
was proved in \cite[Theorem 1]{CM} by operator theoretical arguments. 
The first statement of the following theorem that $\lambda$ is real-analytic in the case of $p=2$
is in \cite[Theorem 7.1]{SW}. 

\begin{theorem}\label{CM}
The map $\Pi|_{IW_p}:IW_p \to W_p$ is real-analytic. Hence,
$$
\lambda=L \circ \Pi \circ L^{-1}|_{iB_p^{\mathbb R}(\mathbb R)^\circ}:iB_p^{\mathbb R}(\mathbb R)^\circ \to B_p^{\mathbb R}(\mathbb R)
$$ 
is also real-analytic. Moreover, $\lambda$ is injective and the inverse $\lambda^{-1}$ is real-analytic.
Namely, $\lambda$ is a real-analytic homeomorphism onto an open subset of $B_p^{\mathbb R}(\mathbb R)$ whose inverse is
also real-analytic. This is true for $\Pi|_{IW_p}$.
\end{theorem}

\begin{proof}
By Corollary 
\ref{real-analytic}, 
$IW_p$ is a real-analytic submanifold of ${\rm WPC}_p$.
Hence, the restriction $\Pi|_{IW_p}$ of the projection $\Pi:{\rm WPC}_p \to W_p$ is real-analytic.
Since $L$ is biholomorphic by 
Theorem \ref{biholo}, 
the conjugate map $\lambda$ is real-analytic.

We will prove the real-analycity of the inverse of $\lambda$.
To this end, we use the corresponding result for the space of chord-arc curves with 
the notation being the same as in \cite{WM-0}.
The definitions of the corresponding subspaces can be found in this paper, but they are not so
essential in the arguments below.
Based on Propositions \ref{pBel} and \ref{derivative}, we see that
the space ${\rm CA}$ of all BMO embeddings with chord-arc images contains ${\rm WPC}_p$.
Then,
there are the inclusion relations of subspaces 
$$W_p \subset {\rm SQS},\ IW_p \subset {\rm ICA},\
RM_p \subset {\rm RM}^\circ 
$$ 
in ${\rm WPC}_p \subset {\rm CA}$. We also have the inclusion relations of subspaces
$$
B_p^{\mathbb R}(\mathbb R) \subset {\rm BMO}_{\mathbb R}^*(\mathbb R),\
iB_p^{\mathbb R}(\mathbb R)^\circ \subset i{\rm BMO}_{\mathbb R}(\mathbb R)^\circ
$$ 
in $B_p(\mathbb R) \subset {\rm BMO}(\mathbb R)$ by Proposition \ref{VMOBp}, where 
${\rm BMO}_{\mathbb R}^*(\mathbb R)$ stands for the space of all real-valued BMO functions $u$
with $e^u$ being an $A_\infty$-weight.

The corresponding map to $\lambda$ between these larger spaces is denoted by
$\tilde \lambda:i{\rm BMO}_{\mathbb R}(\mathbb R)^\circ \to {\rm BMO}_{\mathbb R}^*(\mathbb R)$.
Then, $\lambda=\tilde \lambda|_{iB_p^{\mathbb R}(\mathbb R)^\circ}$.
It is known that $\tilde \lambda$ is a real-analytic homeomorphism onto an open subset of 
${\rm BMO}_{\mathbb R}^*(\mathbb R)$ whose inverse is
also real-analytic (see \cite[Theorem 5]{Se0}).

First of all,
we prove the injectivity of $\lambda$. This is the same as the case of the space of chord-arc curves.
Every $\gamma_0 \in IW_p$ is decomposed uniquely into $\gamma_0=h \circ f$ for $h \in RM_p$ and $f \in W_p$.
Taking the logarithm of the derivative of this equation, we have
\begin{equation*}\label{logder}
\log \gamma_0'= \log h' \circ f+\log f'.
\end{equation*}
Since $\log \gamma_0'=iv$ is purely imaginary and $\log f'$ is real, the real and the imaginary parts of this equation
become
\begin{equation*}\label{ReIm}
0={\rm Re} \log h' \circ f +\log f'\quad {\rm and} \quad v={\rm Im} \log h' \circ f.
\end{equation*}
Moreover, since $\log h'$ is the boundary extension of the holomorphic function $\log H'$ for the Riemann mapping $H$ on $\mathbb U$,
${\rm Re}\log h'$ and ${\rm Im}\log h'$ are related by the Hilbert transformation $\mathcal H$ on $\mathbb R$:
\begin{equation*}
{\rm Im}\log h'={\mathcal H}({\rm Re}\log h').
\end{equation*}
Then, the combination of these equations yields that
\begin{equation}\label{injective}
-P_f\circ {\mathcal H} \circ P_f^{-1}(\log f')=v.
\tag{$\ast$}
\end{equation}
This shows that $v$ is determined by $f$ and thus $\lambda:\log \gamma_0' \mapsto \log f'$ is injective.

\begin{claim}\label{inverse}
Suppose that we have the decomposition
$\tilde \gamma_0=\tilde h \circ \tilde f$ of $\tilde \gamma_0 \in {\rm ICA}$ 
by $\tilde h \in {\rm RM}^\circ$ and $\tilde f \in {\rm SQS}$. In this situation, if $\tilde f$
belongs to $W_p$, then we obtain $\tilde h \in RM_p$ and $\tilde \gamma_0 \in IW_p$.
\end{claim}

\begin{proof}
The formula corresponding to \eqref{injective} reads as
\begin{equation}\label{injective2}
-P_{\tilde f}\circ {\mathcal H} \circ P_{\tilde f}^{-1}(\log \tilde f')=\tilde v.
\tag{$**$}
\end{equation}
Here, $\tilde f \in W_p$ implies $\log \tilde f' \in B_p^{\mathbb R}(\mathbb R)$.
Moreover, since $P_{\tilde f}$ preserves $B_p^{\mathbb R}(\mathbb R)$ by Proposition \ref{pullback}, we have 
$P_{\tilde f}^{-1}(\log \tilde f') \in B_p^{\mathbb R}(\mathbb R)$.
By Lemma \ref{Hilbert}, the Hilbert transformation $\mathcal H$ maps $B_p^{\mathbb R}(\mathbb R)$
to $B_p^{\mathbb R}(\mathbb R)$.
This implies  
${\mathcal H} \circ P_{\tilde f}^{-1}(\log \tilde f') \in B_p^{\mathbb R}(\mathbb R)$. By applying $P_{\tilde f}$ again,
we see that the left side of \eqref{injective2} is in $B_p^{\mathbb R}(\mathbb R)$, and hence 
$\tilde v \in B_p^{\mathbb R}(\mathbb R)$. Since $\log \tilde \gamma_0'=i \tilde v$, we have 
$\tilde \gamma_0 \in IW_p$ and thus $\tilde h=\tilde \gamma_0 \circ \tilde f^{-1} \in RM_p$.
\end{proof}

By the conjugation of $L$,
this claim is equivalent to saying that if $\tilde \lambda(w) \in B_p^{\mathbb R}(\mathbb R)$ for 
$w \in i{\rm BMO}_{\mathbb R}(\mathbb R)^\circ$ then $w \in iB_p^{\mathbb R}(\mathbb R)^\circ$.

We move to the investigation of the derivatives of $\lambda$ and $\tilde \lambda$.
We note the following two facts:
(1) As $\lambda$ is real-analytic, the derivative $d_w\lambda:iB_p^{\mathbb R}(\mathbb R) \to B_p^{\mathbb R}(\mathbb R)$
is a bounded linear operator at every point $w$ of the domain of $\lambda$;
(2) As $\tilde \lambda$ is real-analytic and $\tilde \lambda^{-1}$ is also
real-analytic, the derivative $d_{\tilde w}\tilde \lambda:i{\rm BMO}_{\mathbb R}(\mathbb R) \to {\rm BMO}_{\mathbb R}(\mathbb R)$
is a surjective bounded linear isomorphism at every point $\tilde w$ of the domain of $\tilde \lambda$.

\begin{claim}\label{restriction}
$d_{w} \tilde \lambda|_{iB_P^{\mathbb R}(\mathbb R)}=d_{w} \lambda$
at every point $w$ of the domain $iB_p^{\mathbb R}(\mathbb R)^\circ$ of $\lambda$.
\end{claim}

\begin{proof}
Take any $iv \in iB_P^{\mathbb R}(\mathbb R)$, and set
$d_{w} \lambda(iv)=u$ and $d_{w} \tilde \lambda(iv)=\tilde u$. Then,
$$
\lim_{t \to 0} \left \Vert \frac{\lambda(w+tiv)-\lambda(w)}{t}-u \right \Vert_{B_p} =0; \quad
\lim_{t \to 0} \left \Vert \frac{\tilde \lambda(w+tiv)-\tilde \lambda(w)}{t}-\tilde u \right \Vert_* =0.
$$
Since $w+tiv \in iB_p^{\mathbb R}(\mathbb R)^\circ$ for all $t \in \mathbb R$ sufficiently close to $0$,
we have $\tilde \lambda(w+tiv)=\lambda(w+tiv)$ as well as $\tilde \lambda(w)=\lambda(w)$.
Combined with the estimate of the norms $\Vert \cdot \Vert_* \leq \Vert \cdot \Vert_{B_p}$ by
Proposition \ref{VMOBp},
these two limits imply $u=\tilde u$. Hence, $d_{w} \tilde \lambda(iv)=d_{w} \lambda(iv)$
for every $iv \in iB_P^{\mathbb R}(\mathbb R)$, that is, $d_{w} \tilde \lambda|_{iB_P^{\mathbb R}(\mathbb R)}=d_{w} \lambda$.
\end{proof}

Property (2) as above implies that $d_{\tilde w}\tilde \lambda $ is injective at every $\tilde w$. 
Then, $d_{w} \lambda$ is also injective at every $w \in iB_p^{\mathbb R}(\mathbb R)^\circ$ by Claim \ref{restriction}.
Next, we will show that $d_{w}\lambda$ is surjective. After this, we see from the open mapping theorem that
$d_{w}\lambda:iB_p^{\mathbb R}(\mathbb R) \to B_p^{\mathbb R}(\mathbb R)$ is a bounded linear isomorphism.
Under this condition, the inverse mapping theorem implies that $\lambda^{-1}$ is real-analytic
in some neighborhood of any point in the image $\lambda(iB_P^{\mathbb R}(\mathbb R)^\circ)$, and thus $\lambda^{-1}$ is
globally real-analytic on $\lambda(iB_P^{\mathbb R}(\mathbb R)^\circ)$ which is an open subset of $B_P^{\mathbb R}(\mathbb R)$.

The remaining task is to show that $d_w\lambda$ is surjective at every $w \in iB_p^{\mathbb R}(\mathbb R)^\circ$. 
We take any
tangent vector $u \in B_p^{\mathbb R}(\mathbb R)$ at $\lambda(w)$, and consider a segment 
$\{\lambda(w)+tu\} \subset B_p^{\mathbb R}(\mathbb R)$ with $t$ in a sufficiently small interval $[0,\varepsilon]$.
Since $\lambda(w) \in \tilde \lambda(i{\rm BMO}_{\mathbb R}(\mathbb R)^\circ)$ and 
$\tilde \lambda(i{\rm BMO}_{\mathbb R}(\mathbb R)^\circ)$ is open, we may assume that 
$\{\lambda(w)+tu\} \subset \tilde \lambda(i{\rm BMO}_{\mathbb R}(\mathbb R)^\circ)$.
Then, the inverse image $\tilde \lambda^{-1}\{\lambda(w)+tu\}$ of the segment is
a real-analytic curve $\beta(t)$ in $i{\rm BMO}_{\mathbb R}(\mathbb R)^\circ$ starting at $w=\beta(0)$.
The tangent vector $iv=\frac{d}{dt}\beta(t)|_{t=0}$ of $\beta(t)$ at $t=0$ satisfies
$d_{w}\tilde \lambda(iv)=u$.
On the other hand, since $\{\lambda(w)+t\xi\}$ is contained in $B_p^{\mathbb R}(\mathbb R)$, Claim \ref{inverse}
implies that $\beta(t)$ is contained in $iB_p^{\mathbb R}(\mathbb R)^\circ$. Hence, $iv \in iB_p^{\mathbb R}(\mathbb R)$. 
Then by Claim \ref{restriction},
we have $d_{w}\lambda(iv)=u$.
This shows that $d_{w} \lambda$ is surjective. 
\end{proof}

In this Weil--Petersson curve version of the Coifman--Meyer theorem, we can also ask a question about
the characterization of the domains 
$iB_p^{\mathbb R}(\mathbb R)^\circ$ and $\lambda(iB_p^{\mathbb R}(\mathbb R)^\circ)$ which are contractible and
real-analytically equivalent to each other. The contractibility will be seen by Proposition \ref{continuity}.

\section{The topological group structure and its applications}
For further investigation, we show and use the following fact.

\begin{theorem}\label{topgroup}
The $p$-Weil--Petersson Teichm\"uller space $T_p$ for $p \geq 1$ is a topological group
under the operation $\ast$.
\end{theorem}

For $p=2$, this was proved in \cite[Chap.1, Theorem 3.8]{TT}. 
A similar argument to this case using the quasiconformal extension
and estimating the integral of the complex dilatation
also works for any $p \geq 1$. For Theorem \ref{topgroup}, 
we show the following basic result by using bi-Lipschitz quasiconformal extension.

\begin{lemma}\label{nearid}
If $[\mu]$ and $[\nu]$ converge to $[\rm 0]$ in $T_p$ for $p \geq 1$, 
then $[\mu] \ast [\nu] \to [\rm 0]$ and $[\nu]^{-1} \to [\rm 0]$.
\end{lemma}

\begin{proof}
Let $F:\mathbb U \to \mathbb U$ be a quasiconformal homeomorphism with its complex dilatation 
$\mu \in \mathcal M_p(\mathbb U)$ in the equivalence class $[\mu]$. We may choose $\mu$ so that
$\Vert \mu \Vert_p \to 0$ and $\Vert \mu \Vert_\infty \to 0$ as $[\mu] \to [\rm 0]$.
Let $H:\mathbb U \to \mathbb U$
be a bi-Lipschitz diffeomorphism 
with its complex dilatation 
$\nu \in \mathcal M_p(\mathbb U)$ in the equivalence class $[\nu]$. 
The existence of such an extension is guaranteed by \cite[Lemma 3.4]{WM-1}.
This also implies that
$\Vert \nu \Vert_p \to 0$ and $\Vert \nu \Vert_\infty \to 0$ as $[\nu] \to [\rm 0]$
and that the bi-Lipschitz constant $L \geq 1$ of $H$ is uniformly bounded throughout when $[\nu]$ tends to $[\rm 0]$.
We use the chain rule for the complex dilatation:
$$
\mu \ast \nu(z)=\frac{(\mu \circ H(z))\cdot \frac{\overline{H_z}}{H_z}+\nu(z)}{1+\mu \circ H(z) 
\cdot {\bar \nu}(z)\frac{\overline{H_z}}{H_z}}; \quad
\nu^{-1}(H(z))=-\nu(z)\frac{\overline{H_z}}{H_z}.
$$

For the composition, we estimate the integral as
\begin{equation*}
\begin{split}
\int_{\mathbb U}|\mu \ast \nu(z)|^p\frac{dxdy}{y^2}\leq
\frac{2^{p-1}}{(1-\Vert \mu \Vert_\infty \Vert \nu \Vert_\infty)^p} 
\left(\int_{\mathbb U}|\mu \circ H(z)|^p\frac{dxdy}{y^2}+\int_{\mathbb U}|\nu(z)|^p\frac{dxdy}{y^2}\right).
\end{split}
\end{equation*}
Since $H^{-1}$ is Lipschitz with the constant $L$ in the hyperbolic metric, we have
\begin{equation*}
\begin{split}
\int_{\mathbb U}|\mu \circ H(z)|^p\frac{dxdy}{y^2}=\int_{\mathbb U}|\mu(\zeta)|^p\frac{d \xi d \eta}{J_H(z)y^2}
\leq KL^2 \int_{\mathbb U} |\mu(\zeta)|^p\frac{d \xi d \eta}{\eta^2}
\end{split}
\end{equation*}
for the Jacobian determinant $J_H$ of $H$ and the maximal dilatation $K \geq 1$ of $H$. Thus,
$$
\Vert \mu \ast \nu \Vert_p \leq 2^{\frac{p-1}{p}}\frac{(KL^2\Vert \mu \Vert^p_p+\Vert \nu \Vert^p_p)^{1/p}}{1-\Vert \mu \Vert_\infty \Vert \nu \Vert_\infty}; \quad \Vert \mu \ast \nu \Vert_\infty \leq \frac{\Vert \mu \Vert_\infty +\Vert \nu \Vert_\infty}{1-\Vert \mu \Vert_\infty \Vert \nu \Vert_\infty}.
$$
This implies that $[\mu] \ast [\nu] \to [\rm 0]$ as $[\mu] \to [\rm 0]$ and $[\nu] \to [\rm 0]$.

For the inverse operation, we obtain similarly 
$$
\int_{\mathbb U}|\nu^{-1}(z)|^p\frac{dxdy}{y^2}\leq \int_{\mathbb U}|\nu \circ H^{-1}(z)|^p\frac{dxdy}{y^2}
\leq KL^2 \int_{\mathbb U} |\nu(\zeta)|^p\frac{d \xi d \eta}{\eta^2}
$$
since $H$ is Lipschitz with the constant $L$. Thus,
$$
\Vert \nu^{-1} \Vert_p \leq (KL^2)^{1/p} \Vert \nu \Vert_p; \quad \Vert \nu^{-1} \Vert_\infty=\Vert \nu \Vert_\infty.
$$
This implies that $[\nu]^{-1} \to [\rm 0]$ as $[\nu] \to [\rm 0]$.
\end{proof}

\begin{proof}[Proof of Theorem \ref{topgroup}]
Lemma \ref{nearid} implies that $T_p$ is a partial topological group. 
We have already seen that that the right translation 
$R_{[\nu]}$ for any $[\nu] \in T_p$ is continuous (in fact, biholomorphic) at the beginning of Section 4.
Hence, in order to show that $T_p$ is a topological group according to \cite[Lemma 1.1]{GS}, 
we have only to prove that the adjoint map $T_p \to T_p$ 
defined by $[\mu] \ast[\nu] \ast [\mu]^{-1}$ for any fixed $[\mu] \in T_p$
is continuous as $[\nu] \to [0]$.
The arguments are essentially the same as in \cite[Chapter 1, Lemma 3.5]{TT}. We omit the details here. 
\end{proof}

There are several consequences from Theorem \ref{topgroup}.
We first introduce the curve theoretical product structure for the space ${\rm WPC}_p$ of normalized $p$-Weil--Petersson curves, 
and verify the topological equivalence of those real and complex Banach manifold structures on 
${\rm WPC}_p$ and its subspaces for $p>1$.
Then, the continuity of Riemann mappings defined by
$p$-Weil--Petersson curves, and
the boundedness of the image of a bounded set under the biholomorphic homeomorphism $L$ on 
${\rm WPC}_p$ are proved.

Theorem \ref{topgroup} in particular implies that $R_{[\nu]}([\mu])=[\mu] \ast [\nu]$ 
and $(R_{[\nu]})^{-1}([\mu])=[\mu] \ast [\nu]^{-1}$ are 
continuous with respect to two variables $([\mu],[\nu]) \in T_p \times T_p$. 
Since $\widetilde R_{[\nu]}$ and $Q_{L([\nu],[\nu])}$ correspond as in Proposition \ref{Lequation},
this fact yields the following consequence on the bijection $J$ given in Lemma \ref{arclength}.

\begin{theorem}\label{Jhomeo}
The bijection $J:B_p^{\mathbb R}(\mathbb R) \times iB_p^{\mathbb R}(\mathbb R)^{\circ} \to
L({\rm WPC}_p)$ is a homeomorphism for $p>1$.
\end{theorem}

\begin{proof}
We prove that $Q_u(w)$ and $Q_u^{-1}(w)$ are continuous for $(u,w) \in B_p^{\mathbb R}(\mathbb R) \times L({\rm WPC}_p)$. 
By Proposition \ref{Lequation} with $L([\nu],[\nu])=u$ for some $[\nu] \in T_p$,
we obtain that
$$
Q_{u}(w)=L \circ \widetilde R_{[\nu]} \circ L^{-1}(w).
$$
Then, Theorems \ref{biholo} and \ref{topgroup} imply that this is continuous. Similarly, $Q_u^{-1}(w)$
is also continuous. Since 
$J(u,iv)=Q_{u}(iv)$ and $J^{-1}(u+iv)=(u,Q_u^{-1}(u+iv))$, both $J$ and $J^{-1}$ are continuous.
\end{proof}

\begin{remark}
The continuity of $Q_u(w)$ also implies that the variable change operator 
$P:B_p(\mathbb R) \times W_p \to B_p(\mathbb R)$ defined by $P_h(w)=w \circ h$ is continuous
for both $w \in B_p(\mathbb R)$ and $h \in W_p$. Indeed, $P_h(w)=Q_{\log h'}(w)-\log h'$.
For $p=2$, the corresponding claim to this continuity is in \cite[Proposition 7.1]{SW} for instance.
\end{remark}

\begin{corollary}\label{structures}
The space ${\rm WPC}_p$ of all normalized $p$-Weil--Petersson embeddings is endowed with the real Banach manifold structure
modeled on $B_p^{\mathbb R}(\mathbb R) \times iB_p^{\mathbb R}(\mathbb R)^{\circ}$ by $J^{-1} \circ L$, which is
topologically equivalent to its complex Banach manifold structure $T_p(\mathbb U) \times T_p(\mathbb L)$.
\end{corollary}

Next, we investigate the continuity of Riemann mappings defined by $p$-Weil--Peters\-son curves
with respect to the topological structure of ${\rm WPC}_p$.
As we have seen in Corollary \ref{structures}, the complex-analytic structure $T_p(\mathbb U) \times T_p(\mathbb L)$
and the real-analytic structure $B_p^{\mathbb R}(\mathbb R) \times iB_p^{\mathbb R}(\mathbb R)^{\circ}$ on 
${\rm WPC}_p$ are topologically equivalent. Thus, we can use either of these structures in order to consider the continuity.

\begin{proposition}\label{continuity}
The map $\Phi:{\rm WPC}_p \to RM_p \cong T_p^{RM}$ is a continuous surjection for $p \geq 1$. Moreover, 
$\Phi|_{IW_p}:IW_p \to RM_p$ is a homeomorphism for $p>1$.
\end{proposition}

\begin{proof}
The first statement follows from Theorem \ref{topgroup}.
For $(\Phi|_{IW_p})^{-1}:RM_p \to IW_p$, we consider its conjugate by $J^{-1} \circ L$.
This is nothing but the projection  
$B_p^{\mathbb R}(\mathbb R) \times iB_p^{\mathbb R}(\mathbb R)^\circ \to iB_p^{\mathbb R}(\mathbb R)^\circ$
to the second factor by Proposition \ref{fibers}. Hence, $(\Phi|_{IW_p})^{-1}$ is continuous and $\Phi|_{IW_p}$ is a homeomorphism.
\end{proof}

\begin{remark}
It is proved in \cite[Theorem 2.4]{SW} that the map $iB_2^{\mathbb R}(\mathbb R)^{\circ} \to T_2^{RM}$ 
is a homeo\-morphism,
which corresponds to our $\Phi|_{IW_p}$ under $IW_2 \cong iB_2^{\mathbb R}(\mathbb R)^{\circ}$.
In our framework, $IW_2$ and $iB_2^{\mathbb R}(\mathbb R)^{\circ}$ are real-analytically equivalent under $L$,
so this claim can be rephrased as just the continuity of $\Phi$ as in Proposition \ref{continuity}.
\end{remark}

Thus, we can conclude that $IW_p$ and $RM_p$, both of which can be regarded as
the space of the images of all
$p$-Weil--Petersson embeddings, are naturally endowed with the two analytic structures 
in the following sense.

\begin{theorem}\label{IWp}
The real-analytic submanifold $IW_p$ and the complex-analytic submanifold $RM_p$ of ${\rm WPC}_p$ for $p>1$
are equipped with both the complex-analytic structure of $T_p$
and the real-analytic structure of $iB_p^{\mathbb R}(\mathbb R)^\circ$, which are topologically equivalent.
\end{theorem}

By $J^{-1} \circ L$, we can introduce the product structure $W_p \times IW_p$ to ${\rm WPC}_p$ from
$B_p^{\mathbb R}(\mathbb R) \times iB_p^{\mathbb R}(\mathbb R)^{\circ}$. Both the products $W_p \times RM_p$ and
$W_p \times IW_p$ are homeomorphic to ${\rm WPC}_p$ for $p>1$.
We summarize the correspondence of
the product structures 
$W_p \times RM_p$ and $W_p \times IW_p$ on ${\rm WPC}_p$  
more precisely than Proposition \ref{fibers} by adding the homeomorphic property of the sections for 
the projection $\Phi$.
The proof is similar.
The fiber structure of the projection to the second factors
in the both products are preserved
since such a fiber
consists of $p$-Weil--Petersson curves of the same image.

\begin{proposition}\label{ab}
$(1)$ For every $\gamma_0 \in IW_p$, the projection $\Phi:{\rm WPC}_p \to RM_p$ restricted to $W_p \times \{\gamma_0\} \subset W_p \times IW_p$ 
is a constant map, and hence 
$W_p \times \{\gamma_0\}$ is the fiber of $\Phi$ over $\Phi(\gamma_0)$. 
$(2)$ For every $f \in W_p$, $\Phi$ restricted to
$\{f\} \times IW_p \subset W_p \times IW_p$ is a homeomorphism onto $RM_p \cong T_p^{RM}$, and hence
$\{f\} \times IW_p$ is the section of $\Phi$ through $f$.
\end{proposition}

We remark that in the comparison of the product structures $W_p \times RM_p$ and $W_p \times IW_p$ on ${\rm WPC}_p$,
the fibers to the projection to the first factors are not preserved. The change of the representatives 
$\Pi|_{IW_p}:IW_p \to W_p$ considered in Theorem \ref{CM}
measures the difference between the fibers $RM_p$ and $IW_p$ over the origin.

Finally,
we add one more property to
the biholomorphic mapping 
$L:{\rm WPC}_p \to B_p(\mathbb R)$ restricted to $W_p$,
which is concerning the correspondence of bounded subsets. 
Here, the boundedness in $W_p$ is considered regarding
a metric structure of $W_p \cong T_p^{W}$.
The invariant metric provided for $T_p$ is the $p$-Weil--Petersson metric (see \cite{Cu}, \cite{Mat}, \cite{TT}), and
let $d_p$ denote the $p$-Weil--Petersson distance in $T_p$. This can be defined for $p \geq 1$.

The correspondence $f \mapsto \log f'$ for $f \in W_p$ gives
a real-analytic equivalence of 
the $p$-Weil--Petersson Teichm\"uller space $T_p$ to $B_p^{\mathbb R}(\mathbb R)$
(see Corollary \ref{real-analytic}). Translating Lemma \ref{nearid}  to $B_p^{\mathbb R}(\mathbb R)$ for $p>1$,
we see that $\Vert \log (f\circ h)' \Vert_{B_p} \to 0$ as $\Vert \log f' \Vert_{B_p} \to 0$ and
$\Vert \log h' \Vert_{B_p} \to 0$ for $f, h \in W_p$. 
Extending this consequence to a claim for the variable change operator $P_h$, we obtain the following proposition.
For BMO norm, an analogues result was stated in \cite[p.18]{CM}.
We have seen that $P_h$ is a bounded linear operator in Proposition \ref{pullback}.

\begin{proposition}\label{uniformPh}
For any $h \in W_p$,
let $P_h:B_p(\mathbb R) \to B_p(\mathbb R)$ be the bounded linear operator defined by
$w \mapsto w \circ h$ for every $w \in B_p(\mathbb R)$. Then, there exist 
constants $\tau_0>0$ and $C_0>0$ such that
the operator norm of $P_h$ satisfies $\Vert P_h \Vert \leq C_0$
for every $h \in W_p$ with $\Vert \log h' \Vert_{B_p} \leq \tau_0$.
\end{proposition}

\begin{proof}
Lemma \ref{nearid} implies that there exist $\delta_0 >0$ and $\tau_0>0$
such that if $\Vert \log f' \Vert_{B_p} \leq \delta_0$ and
$\Vert \log h' \Vert_{B_p} \leq \tau_0$, then $\Vert \log (f\circ h)' \Vert_{B_p} \leq 1$.
We may choose $\tau_0$ so that $\tau_0 \leq 1$. Then, for any $u=\log f' \in B_p^{\mathbb R}(\mathbb R)$
with $\Vert u \Vert_{B_p} \leq \delta_0$ and $h \in W_p$ with $\Vert \log h'\Vert_{B_p} \leq \tau_0$,
we have
$$
\Vert P_h(u) \Vert_{B_p} \leq \Vert \log(f \circ h)'\Vert_{B_p}+\Vert \log h'\Vert_{B_p}\leq 2.
$$
For $w=u+iv \in B_p(\mathbb R)$ with $\Vert w \Vert_{B_p} = \delta_0$, we apply this estimate to
$\Vert P_h(u) \Vert_{B_p}$ and $\Vert P_h(v) \Vert_{B_p}$; we obtain that
if $\Vert w \Vert_{B_p} = \delta_0$ and $\Vert \log h'\Vert_{B_p} \leq \tau_0$, then 
$\Vert P_h(w) \Vert_{B_p} \leq 4$. 
This completes the proof by setting $C_0=4/\delta_0$.
\end{proof}

\begin{lemma}\label{general}
For any $h \in W_p$,
the operator norm $\Vert P_h \Vert$ of the variable change operator $P_h:B_p(\mathbb R) \to B_p(\mathbb R)$
depends only on the $p$-Weil--Petersson distance $d_p(h,\rm id)$ for $p>1$.
\end{lemma}

\begin{proof}
For the constant $\tau_0$ in Proposition \ref{uniformPh}, we choose a constant $r_0>0$ such that
if $h \in W_p$ satisfies $d_p(h,{\rm id}) \leq r_0$ then $\Vert \log h' \Vert_{B_p} \leq \tau_0$.
Any element $h \in W_p$ can be joined to $\rm id$ by a curve in $W_p$
with its length arbitrary close to $d_p(h,{\rm id})$. We choose the minimal number of consecutive points
$$
{\rm id}=h_0, h_1, \ldots, h_n=h
$$
on the curve such that $d_p(h_i,h_{i-1}) < r_0$ for any $i=1,\ldots,n$. Then, the number $n$ is determined by
$d_p(h,\rm id)$, and the invariance of $d_p$ under the right translation implies that the composition
$h_i \circ h_{i-1}^{-1}$ satisfies 
$d_p(h_i \circ h_{i-1}^{-1},{\rm id}) <r_0$, and hence $\Vert \log (h_i \circ h_{i-1}^{-1})'\Vert_{B_p} \leq \tau_0$.
By decomposing $h$ into these $n$ mappings, we have
$$
P_h=P_{h_1 \circ h_0^{-1}} \circ P_{h_2 \circ h_1^{-1}} \circ \cdots \circ P_{h_n \circ h_{n-1}^{-1}}.
$$
Then, Proposition \ref{uniformPh} shows that $\Vert P_h \Vert \leq C_0^n$.
\end{proof}

\begin{remark}
For $p=2$, this claim follows from the stronger result mentioned in the remark after Proposition \ref{pullback}.
This is because the Teichm\"uller distance $d_\infty$ is bounded by a certain multiple of 
the $p$-Weil--Petersson distance $d_p$, that is, $d_\infty \lesssim d_p$. See \cite[Proposition 6.10]{Mat1}.
\end{remark}

\begin{theorem}\label{b-b}
For a bounded subset $W \subset W_p$ for $p>1$, the image $L(W)$ is
also bounded in $B_p^{\mathbb R}(\mathbb R)$. In more details, if $h \in W_p$ is
within distance $r$ from $\rm id$ in the $p$-Weil--Petersson distance $d_p$, then $\Vert \log h' \Vert_{B_p}$
is bounded by a constant depending only on $r$. 
\end{theorem}

\begin{proof}
For $h \in W_p$ with $d_p(h,{\rm id}) \leq r$, we decompose $h$ as in the proof of
Lemma \ref{general}. Then, by this lemma, we have
\begin{align*}
\Vert \log h'\Vert_{B_p} 
&=\Vert \log ((h_n \circ h_{n-1}^{-1}) \circ (h_{n-1} \circ h_{n-2}^{-1}) \circ \cdots \circ (h_{1} \circ h_{0}^{-1}) )'\Vert_{B_p}\\
& \leq \Vert P_{h_1 \circ h_0^{-1}} \circ P_{h_2 \circ h_1^{-1}} \circ \cdots \circ P_{h_{n-1} \circ h_{n-2}^{-1}}(\log (h_n \circ h_{n-1}^{-1})')\Vert_{B_p}\\
&\quad + \cdots +\Vert \log(h_{1} \circ h_{0}^{-1})' \Vert_{B_p}\\
& \leq C_0^{n-1} \tau_0+ C_0^{n-2} \tau_0 + \cdots +\tau_0.
\end{align*}
Since $n$ depends only on $r$, the statement is proved.
\end{proof}

Developing Theorem \ref{b-b} further,
we expect (a) for any bounded subset 
$\widetilde W \subset {\rm WPC}_p$, the image $L(\widetilde W)$ is
bounded in $B_p(\mathbb R)$; (b) conversely, for any bounded subset $B \subset B_p^{\mathbb R}(\mathbb R)$,
the inverse image $L^{-1}(B)$ is bounded in $W_p$. 

Concerning statement (a),
the boundedness on ${\rm WPC}_p \cong T_p(\mathbb U) \times T_p(\mathbb L)$ is defined by
the product of the $p$-Weil--Petersson metrics. In order to prove (a), it suffices to show the 
boundedness of the image of a bounded set under the embedding $T_p \to \mathcal B_p$ into the pre-Schwarzian derivative model.
This is seen by the representation of $L$ in the proof of Theorem \ref{holo}. For the
Bers embedding $T_p \to A^p$ (see Appendix), this follows from \cite[Proposition 8.4]{Mat}, but the embedding into $\mathcal B_p$
seems more involved. 

Concerning statement (b), 
if we could prove that $\Vert \mu_u \Vert_p$ and 
$\Vert \mu_u \Vert_\infty<1$ are dominated by $\Vert u \Vert_{B_p}$
for the complex dilatation $\mu_u$ of a certain quasiconformal extension $F_u:\mathbb U \to \mathbb U$ 
of $\gamma_u:\mathbb R \to \mathbb R$,
we would obtain the result from \cite[Theorem 5.4]{Mat}.
However,
we only have the estimates of $\Vert \mu_u \Vert_p$ and $\Vert \mu_u \Vert_\infty$ in terms of $\Vert u \Vert_{B_p}$
in the case where $F_u$ is the variant of the Beurling--Ahlfors extension by the heat kernel and
$\Vert u \Vert_{B_p}$ is sufficiently small (see \cite[Proposition 3.5,
Lemma 4.2]{WM-3}).

\section{Appendix: The pre-Schwarzian derivative model on the upper half-plane}

In this appendix, we provide a complex Banach manifold structure for 
the $p$-Weil--Peters\-son Teichm\"uller space $T_p=T_p(\mathbb U)$ by using
pre-Schwarzian derivatives on $\mathbb U$. This is well-known in the case of
the unit disk $\mathbb D$ for $p=2$. However, since the pre-Schwarzian derivative is not
M\"obius invariant, we carefully treat the case of $\mathbb U$. We will see below that there is a certain advantage of
considering the pre-Schwarzian derivative model on $\mathbb U$ compared with $\mathbb D$.
The generalization to any $p > 1$ is also mentioned.
We note that if we use the Schwarzian derivative model, there is no difference between $\mathbb U$ and $\mathbb D$
due to its M\"obius invariance, and $T_p$ is equipped with
the complex Banach manifold structure for $p \geq 1$. 

Let $A^\infty(\mathbb U)$ denote the Banach space of functions $\varphi$ holomorphic on $\mathbb U$ with norm 
$$
\Vert \varphi \Vert_{A^\infty} = \sup_{z \in \mathbb U} |\varphi(z)|y^2. 
$$
For $p \geq 1$, we also denote by $A^p(\mathbb U)$ the Banach space of holomorphic functions $\varphi$ on $\mathbb U$ with norm
$$
\Vert \varphi \Vert_{A^p} = \left(\frac{1}{\pi}\iint_{\mathbb U} |\varphi(z)|^p y^{2p-2} dxdy \right)^{\frac{1}{p}}. 
$$

For any locally univalent function $g$, the derivative of the logarithm $\mathcal L_g$ and  
the Schwarzian derivative $\mathcal S_g$ are defined respectively by
\begin{equation*}\label{SL}
\mathcal L_g = \log g', \quad \mathcal S_g= \mathcal L_g'' - \frac{1}{2}(\mathcal L_g')^2.
\end{equation*}
The derivative of $\mathcal L_g$ is called the {\it pre-Schwarzian derivative} of $g$. 
We will show the following result on the upper half-plane $\mathbb U$. 
In the case of the unit disk $\mathbb D$, the corresponding theorem for $p \geq 2$ was proved in \cite[Theorems 1, 2]{Gu}.
However, since $\mathcal L_g$ is not M\"obius invariant, 
this is not straightforward from that on $\mathbb D$. As mentioned below, the case of $p=2$ was proved in \cite{STW}.

\begin{theorem}\label{Guo11}
Let $p > 1$. Let $g:\mathbb U \to \mathbb C$ be a conformal mapping on $\mathbb U$ and can be extended to the whole plane $\mathbb C$ quasiconformally  such that
$\lim_{z \to \infty}g(z)=\infty$. Then, the following statements are equivalent $\!:$ 
\begin{enumerate}
\item[(a)]
$g$ extends to a quasiconformal homeomorphism of  $\mathbb C$ whose complex dilatation $\mu$ belongs to $\mathcal M_p(\mathbb L);$
\item[(b)]
$\mathcal L_g \in \mathcal B_p(\mathbb U);$
\item[(c)]
$\mathcal S_g \in A^p(\mathbb U)$.
\end{enumerate}
\end{theorem}

\begin{proof}[Proof of Theorem \ref{Guo11}]
The equivalence of (b) and (c) for $p = 2$ was investigated in \cite[Theorem 4.4]{STW}. Essentially the same argument is valid for general $p > 1$.  
The implication (a) $\Rightarrow$ (c) for $p \geq 1$ is asserted in \cite[Lemma 3.2]{WM-1}, and (c) $\Rightarrow$ (a) for $p \geq 1$ is contained in \cite[Theorem 4.1]{WM-1}.  
The equivalence of (a) and (c)  was proved formerly by \cite[Theorem 2]{Cu} for $p = 2$, and by \cite[Theorems 2]{Gu} and \cite[Theorem 2.1]{ST} for  
$p \geq 2$.
\end{proof} 

Under these preparations, we introduce the pre-Schwarzian derivative model of Teich\-m\"ul\-ler spaces on $\mathbb U$.
Let $g:\mathbb U \to \mathbb C$ be a conformal mapping on $\mathbb U$ satisfying the condition $g(\infty)=\infty$
(i.e. $\lim_{z \to \infty}g(z)=\infty$) that extends to a quasiconformal homeomorphism of 
the whole plane $\mathbb C$. Then, the set $\mathscr T$ of $\mathcal S_g \in A^\infty(\mathbb U)$ for all such $g$ 
is the Schwarzian derivative model of the universal Teichm\"uller space $T$, and
the set $\mathcal T$ of $\mathcal L_g \in {\mathcal B}(\mathbb U)$ for all such $g$ 
is the pre-Schwarzian derivative model of $T$. It is known that $\mathscr T$ is
a bounded domain in $A^\infty(\mathbb U)$ identified with $T$ (the Bers embedding), 
which defines a complex Banach manifold structure for $T$ (see \cite[Section III.4]{Le}).

However, if we do not impose the condition $g(\infty)=\infty$
on $g$ and consider $\mathcal L_g$ for all those $g$, then they are classified into uncountably many components and
$\mathcal T$ is the one containing $0$.
To see this, we consider the conformal mapping $\tilde g=g \circ T^{-1}$ 
of $\mathbb D$ pushed forward by the Cayley transformation $T:\mathbb U \to \mathbb D$.
Since 
$$
\log \tilde g'=T_* \mathcal L_g + \log(T^{-1})'
$$
and $\log(T^{-1})' \in \mathcal B(\mathbb D)$, this defines an affine isometry 
$\mathcal B(\mathbb U) \to \mathcal B(\mathbb D)$. Under this isometric isomorphism,
the components 
in ${\mathcal B}(\mathbb U)$ correspond to
those in ${\mathcal B}(\mathbb D)$ bijectively, which are ${\mathcal T}_\omega(\mathbb D)$ ($\omega \in \mathbb S$)
and ${\mathcal T}_{\rm bdd}(\mathbb D)$ characterized by the property that a point 
$\omega \in \mathbb S$ or no point of $\mathbb S$ is mapped to $\infty$ by 
the conformal mapping $\tilde g:\mathbb D \to \mathbb C$. See \cite{Zhur} and \cite[Section 4]{Sug}.
Then, the component $\mathcal T \subset \mathcal B(\mathbb U)$ containing $0$ corresponds to 
${\mathcal T}_{T(\infty)}(\mathbb D)={\mathcal T}_1(\mathbb D)$,
which is biholomorphically equivalent to the universal Teichm\"uller space $T \cong \mathscr T$.

Theorem \ref{Guo11} guarantees that we can also consider the pre-Schwarzian derivative
model of the $p$-Weil--Petersson Teichm\"uller space $T_p$ for $p > 1$ in the same manner.
Concerning the correspondence between the spaces on $\mathbb U$ and $\mathbb D$, there is
a difference from the case of $T$, which is a fact that $\log(T^{-1})'$ does not belong to $\mathcal B_p(\mathbb D)$.
Nevertheless, Theorem \ref{Guo11} asserts that there is a bijective correspondence
between $\mathcal T_p=\mathcal T \cap \mathcal B_p(\mathbb U)$ and $\mathscr T_p =\mathscr T \cap A^p(\mathbb U)$ 
under the map $\alpha:\mathcal B_p(\mathbb U) \to A^p(\mathbb U)$ 
given by $\alpha(\varphi)=\varphi''-(\varphi')^2/2$ from $\mathcal S_g= \mathcal L_g'' - \frac{1}{2}(\mathcal L_g')^2$. 
It was proved in \cite[Lemma 2.3]{TS} that $\alpha$ is holomorphic in the case of $\mathbb D$ for $p \geq 2$.
This is also true in the case of $\mathbb U$ for $p > 1$.
Moreover, $\mathscr T_p$ is a contractible domain in $A^p(\mathbb U)$ identified with
the $p$-Weil--Petersson Teichm\"uller space $T_p$ for $p > 1$, which provides the complex Banach manifold structure for $T_p$
(see \cite{Cu, Gu, WM-1, Ya}).

We finish our discussion by stating the following theorem.

\begin{theorem}\label{model}
The holomorphic map $\alpha$ restricted to $\mathcal T_p$
is a biholomorphic homeomorphism onto $\mathscr T_p$ for $p > 1$. Hence, the complex Banach manifold structures
on these two models of the
$p$-Weil--Petersson Teichm\"uller space $T_p$ are biholomorphically equivalent. 
\end{theorem}

\begin{proof}
We have seen that $\alpha:\mathcal T_p \to \mathscr T_p$ is a holomorphic bijection.
For a Beltrami coefficient $\mu$ in $\mathcal M_p(\mathbb L)$, 
we take a conformal map $f_\mu$ on $\mathbb U$ with $f_\mu(\infty)=\infty$ that is
quasiconformally extendable to $\mathbb L$ having the complex dilatation $\mu$.
Then, the map $\ell:\mathcal M_p(\mathbb L) \to \mathcal T_p$ defined by $\mu \mapsto \mathcal L_{f_\mu}$
is continuous at $\mu$
such that $f_\mu$ (or $f^\mu$) is bi-Lipschitz on $\mathbb L$. To see this,
even in the case of $\mathbb L$ for $p > 1$,
the proof of \cite[Theorem 2.4]{TS} in the case of $\mathbb D^*$ for $p \geq 2$ can be applied once we fill the step of showing that 
the Bers projection
$\sigma:\mathcal M_p(\mathbb L) \to \mathscr T_p$ given by $\mu \mapsto \mathcal S_{f_\mu}$ is continuous at $\mu$
such that $f_\mu$ is bi-Lipschitz on $\mathbb L$.
This is shown in \cite[Lemma 3.2]{WM-1}.

Moreover, as the composition 
$\alpha\circ \ell$ coincides with 
$\sigma$, at any point $\psi \in \mathscr T_p$,
there is a local continuous right inverse $s$ of $\alpha\circ \ell$ satisfying that
$s(\psi)$ is an arbitrarily given $\mu \in \mathcal M_p(\mathbb L)$ such that $f_\mu$ is bi-Lipschitz on $\mathbb L$ 
(see \cite[Theorem 2.1]{TS}, \cite[Proposition 4.3]{Ya}, and \cite[Theorem 4.1]{WM-1}).
It follows that $\ell \circ s$ becomes a local continuous right inverse of $\alpha$ at $\psi$,
from which we see that $\alpha^{-1}$ is continuous.
By a well-known argument in this situation, the holomorphy of $\alpha^{-1}:\mathscr T_p \to \mathcal T_p$
follows from its continuity. Thus, $\alpha$ is biholomorphic.
\end{proof}

%


\begin{thebibliography}{99}
\bibitem{Ah} L.V. Ahlfors, Conformal Invariants: Topics in Geometric Function Theory, AMS Chelsea Publishing, 2010.
\bibitem{Bi} C.J. Bishop, Weil--Petersson curves, conformal energies, $\beta$-numbers, and minimal surfaces, preprint. 
\bibitem{BS} G. Bourdaud and W. Sickel, Changes of variable in Besov spaces, Math. Nachr. 198 (1999), 19--39.
\bibitem{B2} G. Bourdaud, Changes of variable in Besov spaces II, Forum Math. 12 (2000), 545--563.
\bibitem{Ch} S.B. Chae, Holomorphy and Calculus in Normed Spaces, Pure and Applied Math. 92, Marcel Dekker, 1985.
\bibitem{CM} R.R. Coifman and Y. Meyer, Lavrentiev's curves and conformal mappings, Institute Mittag--Leffler, Report No.5, 1983.
\bibitem{Cu} G. Cui, Integrably asymptotic affine homeomorphisms of the circle and Teichm\"uller spaces, Sci. China Ser. A 43 (2000), 267--279. 
\bibitem{FKP} R.A. Fefferman, C.E. Kenig and J. Pipher, The theory of weights and the Dirichlet problems for elliptic equations, Ann. of Math. 134 (1991), 65--124.
\bibitem{GR} J. Garc\'ia-Cuerva and J.L. Rubio De Francia, Weighted Norm Inequalities and Related Topics, Elsevier, 2011. 
\bibitem{GS} F.P. Gardiner and D. Sullivan, Symmetric structure on a closed curve, Amer. J. Math. 114 (1992), 683--736.
\bibitem{Ga} J.B. Garnett, Bounded Analytic Functions, Academic Press, New York, 1981.
\bibitem{GP} K. Gr\"ochenig and M. Piotrowski, Molecules in coorbit spaces and boundedness of operators, Studia Math. 192 (2009),
61--77.
\bibitem{Gu} H. Guo, Integrable Teichm\"uller spaces, Sci. China Ser. A 43 (2000), 47--58. 
\bibitem{Le} O. Lehto, Univalent Functions and Teichm\"uller Spaces, 
Graduate Texts in Math. 109, Springer, 1987.
\bibitem{Lem} P.G. Lemari\'e, Continuit\'e sur les espaces de Besov des op\'erateurs d\'efinis par des int\'egrales singuli\`eres, Ann. Inst. Fourier (Grenoble) 35 (1985), 175--187.
\bibitem{Leo} G. Leoni, A First Course in Sobolev Spaces, Second Edition, Graduate Studies in Math. 181, American Math. Soc., 2017.
\bibitem{Mat} K. Matsuzaki, Rigidity of groups of circle diffeomorphisms and Teichm\"uller spaces, 
J. Anal. Math. 40 (2020), 511--548.
\bibitem{Mat1} K. Matsuzaki, Circle diffeomorphisms, rigidity of symmetric conjugation and
affine foliation of the universal Teich\-m\"ul\-ler space,
Geometry, dynamics, and foliations 2013,
Advanced Studies in Pure Mathematics vol. 72, pp. 145--180,
Mathematical Society of Japan, 2017.
\bibitem{M} B. Muckenhoupt, Weighted norm inequalities for the Hardy maximal function, Trans. Amer. Math. Soc. 165 (1972),
207--226.
\bibitem{Mu} J. Mujica, Complex Analysis in Banach Spaces, Dover, 2010.
\bibitem{NS} S. Nag and D. Sullivan, Teichm\"uller theory and the universal period mapping via quantum calculus and 
the $H^{1/2}$ space on the circle, Osaka J. Math. 32 (1995), 1--34.
\bibitem{Pa} M. Pavlovi\'c, On the moduli of continuity of $H_p$-functions with $0<p<1$. Proc. Edinburgh Math. Soc.
35 (1992), 89--100. 
\bibitem{Pe} V.V. Peller, Hankel operators of class $\mathfrak{S}_p$ and their applications (rational approximation, Gaussian
processes, the problem of majarizing operators), Math. USSR-Sbornik 41 (1982), 443--479.
\bibitem{Pom75} Ch. Pommerenke, Univalent Functions, Vandenhoeck \& Ruprecht, 1975. 
\bibitem{Pom} Ch. Pommerenke, Boundary Behaviour of Conformal Maps, Springer, 1992.
\bibitem{Rei} A. Reijonen, 
Besov spaces induced by doubling weights, Constr. Approx. 53 (2021), 503--528.
\bibitem{SM} E. Sharon and D. Mumford. 2D-Shape analysis using conformal mapping. Int. J. Comput. Vis. 70 (2006), 55--75.
\bibitem{Se0} S. Semmes, The Cauchy integral, chord-arc curves, and quasiconformal mappings. 
The Bieberbach conjecture (West Lafayette, Ind., 1985), 167--183, Math. Surveys Monogr., 21, Amer. Math. Soc., Providence, RI, 1986.
\bibitem{Se} S. Semmes, Quasiconformal mappings and chord-arc curves, Trans. Amer. Math. Soc. 306 (1988), 233--263.
\bibitem{Sh18} Y. Shen, Weil--Petersson Teichm\"uller space, Amer. J. Math. 140 (2018), 1041--1074. 
\bibitem{ST} Y. Shen and S. Tang, Weil--Petersson Teichm\"uller space II: smoothness of flow curves of $H^{\frac{3}{2}}$-vector fields, Adv. Math. 359 (2020), 106891. 
\bibitem{STW} Y. Shen, S. Tang and L. Wu, Weil--Petersson and little Teichm\"uller spaces on the real line, 
Ann. Acad. Sci. Fenn. Math. 43 (2018), 935--943. 
\bibitem{SW} Y. Shen and L. Wu, Weil--Petersson Teichm\"uller space III: dependence of Riemann mapping for Weil--Petersson
curves, Math. Ann. 381 (2021), 875--904. 
\bibitem{St} E. Stein, Singular Integrals and Differentiability Properties of Functions, Princeton Univ. Press, 1970. 
\bibitem{Sug} T. Sugawa, The universal Teichm\"uller space and related topics,
Proceedings of the international workshop on quasiconformal
mappings and their applications, pp. 261--289, Narosa Publishing House, 2007.
\bibitem{Tang} S. Tang, Some characterizations of the integrable Teichm\"uller space, Sci. China Math. 56 (2013) 541--551.
\bibitem{TS} S. Tang and Y. Shen, Integrable Teichm\"uller space, J. Math. Anal. Appl. 465 (2018), 658--672. 
\bibitem{TT} L. Takhtajan and L.P. Teo, Weil--Petersson metric on the universal Teichm\"uller space, Mem. Amer. Math. Soc. 183(861) (2006).
\bibitem{Vo} S. Vodopyanov, Mappings of homogeneous groups and imbeddings of functional spaces, Siberian Math. J. 30 (1989), 685--698. 
\bibitem{Wa} Y. Wang, Equivalent descriptions of the Loewner energy. Invent. Math. 218 (2019), 573--621.
\bibitem{WM-3} H. Wei and K. Matsuzaki, The $p$-Weil--Petersson Teichm\"uller space and the quasiconformal extension of curves, J. Geom. Anal. 32, 213 (2022).
\bibitem{WM-2} H. Wei and K. Matsuzaki, The VMO-Teichm\"uller space and the variant of Beurling--Ahlfors extension by heat kernel, Math. Z. 302 (2022), 1739--1760. \bibitem{WM-0} H. Wei and K. Matsuzaki, BMO embeddings, chord-arc curves, and Riemann mapping parametrization, arXiv:2111.14312.
\bibitem{WM-1} H. Wei and K. Matsuzaki,  The $p$-integrable Teichm\"uller space for $p \geqslant 1$, arXiv:2210.04720.
\bibitem{Wu} S. Wu, Analytic dependence of Riemann mappings for bounded domains and minimal surfaces,
Comm. Pure Appl. Math. 156 (1993), 1303--1326.
\bibitem{Ya} M. Yanagishita, Introduction of a complex structure on the $p$-integrable Teichm\"uller space,
Ann. Acad. Sci. Fenn. Math. 39 (2014), 947--971.
\bibitem{Zh} K. Zhu, Operator Theory in Function Spaces, Math. Surveys Mono. 138, American Math. Soc., 2007.
\bibitem{Zhur} I.V. Zhuravlev, A model of the universal Teichm\"uller space, Sibirsk. Mat. Zh. 27 (1986), 75--82.
\end{thebibliography}
\end{document}